\documentclass[11pt,twoside,letterpaper]{article} 
\usepackage{times,fancyhdr}
\usepackage[dvips]{graphicx}

\def\figurename{Figure} 
\makeatletter
\renewcommand{\fnum@figure}[1]{\figurename~\thefigure.}
\makeatother

\def\tablename{Table} 
\makeatletter
\renewcommand{\fnum@table}[1]{\tablename~\thetable.}
\makeatother

\usepackage{mathrsfs}
\usepackage{amsmath}
\usepackage{amssymb}
\usepackage{amsfonts}
\usepackage{amsthm,amscd}
\usepackage{xcolor}
\usepackage{anysize,enumerate}

\usepackage{dsfont}
  \usepackage{paralist}
  \usepackage{graphics} 
  \usepackage{epsfig} 
\usepackage{graphicx}
\usepackage{epstopdf}
\usepackage{hyperref}
\usepackage[numbers,sort&compress]{natbib}

\hypersetup{hypertex=true,
            colorlinks=true,
            linkcolor=blue,
            anchorcolor=blue,
            citecolor=blue}

\newtheorem{theorem}{Theorem}[section]
\newtheorem{lemma}[theorem]{Lemma}
\newtheorem{corollary}[theorem]{Corollary}

\theoremstyle{definition}
\newtheorem{definition}[theorem]{Definition}

\theoremstyle{remark}
\newtheorem{remark}[theorem]{Remark}
\theoremstyle{definition}\newtheorem{assumption}{Assumption}[section]
\numberwithin{equation}{section}

\def\h{\mathcal H}
\def\o{\mathcal O}

\def\d{\mathrm d}
\def\q{\mathcal Q}
\def\L{\mathcal L}
\def\e{\epsilon}
\def\tr{\triangle}
\def\r{\rho}
\def\D{\mathcal D}

\begin{document}
\title{\bfseries Continuity of the attractors in time-dependent spaces and applications
  \footnote{*Corresponding author: Zhijian Yang, e-mail: liyn@hrbeu.edu.cn, yzjzzut@tom.com. The first author is supported by the National
Natural Science Foundation  for Young Scientists  of China (No.12101155) and the
Natural Science Foundation of Heilongjiang Provience  of China (No. LH2021A001).
  The second author is supported by the National
Natural Science Foundation of China (No.12171438). } }
\author{Yanan Li$^1$, \ \ Zhijian Yang$^{2, *}$\\
${}^1$ College of Mathematical Sciences, Harbin Engineering University, 150001, China\\
${}^2$ School of Mathematics and Statistics, Zhengzhou University,  450001,   China}
\date{}
\maketitle \thispagestyle{empty} \setcounter{page}{1}

\begin{abstract} In this paper, we investigate the continuity of the attractors in time-dependent phase spaces. (i) We establish  two abstract criteria on the upper semicontinuity and the residual continuity of the  pullback $\mathscr D$-attractor  with respect to the perturbations,  and an equivalence criterion between  their continuity and the pullback equi-attraction, which generalize the
continuity theory of  attractors developed recently in  \cite{Hoang1,Hoang2}  to that in  time-dependent spaces.  (ii)  We propose  the notion  of pullback $\mathscr D$-exponential attractor, which includes the notion of time-dependent exponential attractor \cite{Ly-YJDDE} as its spacial case,   and establish its existence and H\"{o}lder continuity criterion  via quasi-stability method introduced originally by Chueshov
and Lasiecka \cite{Chueshov2008,Chueshov2015}. (iii) We apply  above-mentioned criteria  to the semilinear damped wave equations with perturbed time-dependent speed of propagation: $\e\rho(t) u_{tt}+\alpha u_t -\Delta u+f(u)=g$, with perturbation parameter $\e\in(0, 1]$, to realize  above mentioned continuity of pullback   $\mathscr D$ and $\mathscr D$-exponential attractors in time-dependent phase spaces, and the method developed here allows to  overcome  the difficulty of the hyperbolicity of the model. These results deepen and extend
recent theory of attractors in time-dependent spaces in literatures
 \cite{Di2011,pata,Kloeden2008}.
 \end{abstract}

\vspace{.08in} \noindent \textbf{Keywords}: Time-dependent phase space, pullback $\mathscr D$-attractor, pullback $\mathscr D$-exponential attractor, continuity of attractors, semilinear damped  wave equation.

\vspace{.08in}  \noindent \textbf{AMS subject classifications:} Primary:  37L15, 37L30; Secondary: 35B65, 35B40, 35B41.

\tableofcontents
\section{Introduction}

In the last decade, there have been more and more concern on the  theory of attractors in time-dependent spaces, which can be seen as a development of the theory of nonautonomous infinite dimensional dynamical system. Some key notations such as  pullback $\mathscr D$-attractors, time-dependent global and   exponential attractors in time-dependent spaces have been defined and the related theory including existence theorems has been established,  and there have been many applications in various mathematical physical models (cf. \cite{PataAJM1, PataAJM2,Di2011,Kloeden2008,pata,Ly-YJDDE,PataNARWA}).  In this paper, we shall investigate the theory on the continuity of the pullback $\mathscr D$-attractors, the existence and the continuity of the pullback $\mathscr D$-exponential  attractors and give their application to the semilinear damped wave equation with  perturbed  time-dependent speed of propagation
 \begin{equation}\label{06231}
 \rho_\e (t) u_{tt}+\alpha u_t -\Delta u+f(u)=g,
 \end{equation}
 where $\frac{1}{\rho_\e(t)}=\frac{1}{\e\rho(t)}$ is  called perturbed time-dependent speed of propagation \cite{pata} and the perturbation parameter $\e\in (0, 1]$.

 In 2008,  Kloeden, Mar\'in-Rubio and Real \cite{Kloeden2008} proposed the concept of  pullback $\mathscr D$-attractors for the  process acting on time-dependent phase spaces,  established an abstract criterion on the existence of the  pullback $\mathscr D$-attractors and gave its application to the semilinear heat equation in a non-cylindrical domain. For narrative convenience, we first state a few related definitions.
  \begin{definition}
   A process acting on time-dependent metric spaces $(X_t, d_t)$ is a two-parametrical family of operators $\{U(t, \tau): X_\tau\rightarrow X_t\ |\ \tau\leq t\in \mathbb R\}$ satisfying that
\begin{enumerate}[(i)]
  \item $U(\tau, \tau)$ is the identity mapping on $X_\tau$, $\forall \tau\in \mathbb R$;
  \item $U(t,s)U(s,\tau)=U(t,\tau)$, $-\infty<\tau\leq s\leq t<\infty$.
\end{enumerate}
\end{definition}

Define the   universe
\begin{align}
\mathscr D=\left\{\mathcal D=\{D(t)\}_{t\in \mathbb R}\ |\ \emptyset \neq D(t)\subset X_t, t\in \mathbb R, \ \hbox{and}\ \ \mathcal D\ \hbox{is of some properties}\right\}.
\end{align}

\begin{definition}\label{PDA} A family $\mathcal{A}=\{A(t)\}_{t\in\mathbb R}$ is called a pullback $\mathscr D$-attractor of the   process $U(t, \tau)$ acting on time-dependent metric spaces $(X_t, d_t)$, if\\
(i)\ $A(t)$ is a  compact subset of $X_t$ for each $t\in \mathbb R$;\\
(ii)\ $\mathcal A$ is invariant, that is,
  \begin{equation*}
  U(t, \tau)A(\tau)= A(t), \ \ -\infty<\tau\leq t<+\infty;
  \end{equation*}
(iii)\ $\mathcal A$ is a pullback $\mathscr D$-attracting family, that is, for any $\mathcal D\in \mathscr D$,
      \begin{equation*}
      \lim_{\tau\rightarrow -\infty} \mathrm{dist}_{X_t}\left(U(t, \tau) D(\tau), A(t)\right)=0, \ \ \ \forall t\in \mathbb R,
      \end{equation*}
      hereafter, $\mathrm{dist}_{X_t}(B, C)=\sup_{x\in B}\inf_{y\in C}d_t(x, y)$ denotes the Hausdorff semidistance between two nonempty subsets $B, C$ of $X_t$.

      In addition, the  pullback $\mathscr D$-attractor $\mathcal{A}$ is said to be minimal if  $A(t)\subset C(t)$ for all $t\in \mathbb R$ whenever $\mathcal C=\{C(t)\}_{t\in\mathbb R}$ is  a pullback $\mathscr D$-attracting family of non-empty closed sets.
   \end{definition}

\begin{remark}
 (i) Let  $\{X_t\}_{t\in \mathbb{R}}$ be a family of Banach spaces, $U(t, \tau)\in C(X_\tau, X_t)$,  and for every  $\D\in \mathscr D, \D$  be backward bounded, i.e., $\sup_{s\in (-\infty, t]}\|D(s)\|_{X_s}<+\infty$. Then  the related   pullback $\mathscr D$-attractor $\mathcal A$ becomes the  time-dependent global attractor proposed by Di Plinio et al. \cite{Di2011}.

 (ii)   Let $\{X_t\}_{t\in \mathbb{R}}$ be a family of normed linear spaces,   for every    $\D\in \mathscr D, \D$ be uniformly bounded, i.e., $\sup_{s\in \mathbb{R}}\|D(s)\|_{X_s}<\infty$, $\mathcal A\in \mathscr D$ and replace the invariance (ii) by the minimality in Definition \ref{PDA}. Then  $\mathcal A$ becomes the  time-dependent global attractor   proposed by Conti et al. \cite{pata}, where  the minimality ensures the uniqueness of $\mathcal A$, and  such an attractor is  invariant whenever
the process $U(t, \tau)$ is $T$-closed for some $T > 0$, i.e., $U(t, t-T)$ is closed for all $t\in \mathbb{R}$.
 \end{remark}

    Conti et al. \cite{pata}   developed the  theory of time-dependent global attractor initialed by Di Plinio et al. \cite{Di2011} and gave its application to the  nonautonomous semilinear damped wave equations  with time-dependent speed of propagation
 \begin{equation}\label{06191}
 \rho(t) u_{tt}+\alpha u_t -\Delta u+f(u)=g,
 \end{equation}
 where $\frac{1}{\rho(t)}$ stands for the time-dependent speed of propagation. They  exploited  new framework to
prove that the related process acting on   time-dependent phase spaces $\{\h_t\}_{t\in \mathbb R}$ has an   invariant   time-dependent global attractor $\mathcal A=\{A(t)\}_{t\in \mathbb R}$.  After that,  Conti and Pata \cite{PataNARWA}   supplemented the general theory with
two new results: the first gives the structure of the time-dependent attractor $\mathcal A=\{A(t)\}_{t\in \mathbb R}$ in terms of complete bounded trajectories of
the system; the second  provides
sufficient conditions in order for the section $A(t)$ to be close  to (in terms of Hausdorff semidistance) the global attractor $A_\infty$ of the limiting equation:
 \begin{equation*}
 \alpha u_t -\Delta u+f(u)=g
 \end{equation*}
as $t\rightarrow +\infty$  provided  that $\lim_{t\rightarrow +\infty}\rho(t)=0$,
 precisely,
   \begin{equation*}
 \lim_{t\rightarrow +\infty} \mathrm{dist}_{H^1}\left(\Pi_t A(t), \mathcal A_{\infty}\right)=0
  \end{equation*}
 where $\Pi_t: \h_t\rightarrow H^1$  is the projection on the first component of $\h_t,  \Pi_tA(t)=\{\xi\in H^1| (\xi,\eta)\in A(t)\}$.

 Replacing the weak damping $\alpha u_t $ in model \eqref{06191} by more complex nonlinear one  $\left[1+\rho(t) f'(u) \right]u_t$, Conti and Pata \cite{Conti2015AMC} studied the same issue for the corresponding  one dimensional heat conduction model of Cattaneo type.

These pioneering works promote the development of the attractor theory in time-dependent phase space. Since then,  there are many researches  on  the existence criteria of the pullback $\mathscr D$-attractor or time-dependent global attractor, as well as their applications in different  mathematical physical  models (cf. \cite{Kloeden2009, MTF, Meng2016, Song2019, Sun2015, Xiao2015,Zhou20181} and references therein).

 More recently, Conti et al. \cite{PataAJM1, PataAJM2}  established the well-posedness of   solutions, the existence and the regularity of the invariant time-dependent global attractor $\mathcal A=\{A(t)\}_{t\in \mathbb R}$ for the following   viscoelastic wave model with time-dependent memory kernel $k(t,s)$ in $\Omega\subset \mathbb{R}^3$:
 \begin{equation}\label{06221}
 u_{tt}(t)-[1+k(t,0)]\Delta u(t)-\int_0^\infty \partial_{s}k(t,s)\Delta u(t-s) \mathrm{d}s+f(u(t))=g,
 \end{equation}
 and showed that the section $A(t)$ is close  to (in the sense of the Hausdorff semidistance $\mathrm{dist}_{H^2\times H^1}(\cdot, \cdot)$) the global attractor $\hat{A}$ of the  following  Kelvin-Voigt type model
 \begin{equation*}
u_{tt}(t)-\Delta u(t)-m\Delta u_t(t)+f(u(t))=g
 \end{equation*}
 as $t\rightarrow \infty$ provided that $\lim_{t\rightarrow\infty} k(t,\cdot)=m\delta_0(\cdot)$ (in the distributional sense), where $\delta_0(\cdot)$  denotes the Dirac mass.

 Based on the works of \cite{PataAJM1,PataAJM2}, Li and Yang \cite{Ly-YJDDE} further  presented a notion of time-dependent exponential attractor, provided an abstract    existence criterion and gave its application to the model \eqref{06221} to establish the existence  and  the regularity of the related  time-dependent exponential attractors.

For comparison,
  we    give a short survey for the continuity theory of the global  and exponential attractors for the dynamical system on fixed phase space, i.e.,  $X_t\equiv X,\ \forall t\in \mathbb{R}$.
 One of the most common ways to describe the stability of global attractors is the upper semicontinuity in the sense of Hausdorff semidistance, there have been some abstract criteria and many of their applications to a variety of model equations with various  perturbations (see e.g. \cite{19, 22,Bortolan, 27, 31, Carvalho1, Freitas,150} and references therein). While the lower semicontinuity and therefore  the continuity of the  attractors is  very difficult for it requires strict conditions on the structure of the unperturbed attractor, which are rarely satisfied even for a slightly simpler global attractor of complicated systems (see Hale et al \cite{90,Hale1990},  Stuart and Humphries \cite{Stuart}).

  Recently, some advances on this issue have been archived.  Exploiting the Baire category theorem (cf. \cite{Oxtoby}) instead of   discussing the structure of the attractor, Babin and Pilyugin \cite{Babin1} showed  a residual continuity criterion, that is, the  global attractors are continuous with respect to (w.r.t.) perturbation parameter  $\lambda$ in a residual set of parameter space $\Lambda$.  Under more relaxed conditions, Hoang, Olson and Robinson \cite{Hoang1,Hoang2} simplified the proof of \cite{Babin1}  and established a few  residual continuity criteria of global, pullback and uniform attractors, respectively. Moreover, based on the works in \cite{Kloeden2, Kloeden3}, the authors  \cite{Hoang1,Hoang2} also proved a continuity criterion of the above-mentioned attractors on  $\Lambda$ (rather than a residual subset of $\Lambda$), and showed the equivalence between the continuity and the equi-attraction of the attractors. For the related research on this topic, one can see \cite{Aragao, Babin1,29, LiyangrongJDDE}.

However, as it says in \cite{70, 73}, the  global attractor may
have two  essential drawbacks:  (i) the concept lacks  the description on its fractal dimension; (ii) its attracting rare for the bounded subset in phase space may be
  arbitrarily slow, which leads to that   it is   difficult (if not impossible) to estimate
 in terms of the physical parameters of the system  and even makes it unobservable. Moreover,
 the global attractor is of, in general, upper semi-continuity and residual continuity  w.r.t. perturbations, which cause that  it may  change  drastically in the complementary set of the residual set under very
small perturbations.

The exponential attractor suggested originally in \cite{70}  overcomes these drawbacks of global attractor in  attracting rate, finiteness of fractal dimension  and   stability.  But due to the non-uniqueness of the exponential attractors,  the ``optimal" choice of an exponential attractor which is stable at every point in perturbation  parameter space is very important.

Comparing the attractor theory in time-dependent phase spaces with that in fixed phase space, one sees that the former is more complex and far from complete.  For example, what about
 the upper and lower semicontinuity of the pullback $\mathscr D$-attractors w.r.t. perturbations? Can we give a proper notion of the   pullback $\mathscr D$-exponential attractor and an existence criterion? What about the continuity of the pullback $\mathscr D$-exponential attractor w.r.t. perturbations? What about the applications of these abstract  criteria to the mathematical physical  models? All these questions are  unsolved.

The purpose of this paper is to solve these questions.  The main contributions  are as follows:

 (i)  We establish two abstract criteria on the upper semicontinuity and residual continuity of the pullback $\mathscr D$-attractor in the time-dependent phase spaces, respectively. (see Theorem \ref{31} and Theorem \ref{32})

  (ii) We show the equivalence between the continuity of the pullback $\mathscr D$-attractor w.r.t. the perturbation parameter and its  equi-attraction. (see Theorem \ref{36})

 (iii) We  present the notion of  the pullback $\mathscr D$-exponential attractor in time-dependent phase spaces, and provide  an abstract criterion on its existence and H\"{o}lder continuity via quasi-stability method introduced originally by Chueshov
and Lasiecka \cite{Chueshov2008,Chueshov2015}, which    is a continuation of  the
 researches on  exponential attractors in time-dependent phase spaces in recent literature \cite{Ly-YJDDE}. (see Theorem \ref{42}, Corollary \ref{ea1} and Corollary \ref{ea2})

 (iv) Applying above-mentioned criteria to the perturbed semilinear damped wave equation \eqref{06231} we show that under the same assumptions as in \cite{pata},

  (a) the related evolution process $U_\e (t, \tau)$ has a  pullback $\mathscr D$-attractor  $\mathcal A_\e=\{A_\e(t)\}_{t\in \mathbb R}$  for each $\e\in(0,1]$,  which is  upper semicontinuous  and residual continuous  w.r.t. the perturbation parameter $\e$, respectively; (see Theorem \ref{52})

   (b)   for every $\e_0\in (0, 1]$, there exists a family of   pullback $\mathscr D$-exponential attractors $\mathcal E_\e=\{E_\e(t)\}_{t\in \mathbb R}$ depending on $\e_0$, which is  H\"{o}lder continuous  at the  point $\e_0$. (see Theorem \ref{54})

 The features of these results  are that:

 (i) In the upper semicontinuity and the residual continuity criteria of the pullback $\mathscr D$-attractor,  we replace the  uniform requirement for all $t\in \mathbb R$  appearing  in the corresponding   criteria in fixed phase space with the relaxed backward uniform requirement for $t\leq t_0$, and remove the requirement for the uniform compactness of the sections of attractors as in literatures \cite{Hoang2},  which made  them being  more applicable.  In addition, we show the equivalence between the continuity and the equi-attraction of the pullback $\mathscr D$-attractor w.r.t.  perturbation parameter.

 (ii) In the  existence and  continuity criterion of the pullback $\mathscr D$-exponential attractors in  time-dependent phase spaces,  we replace the Banach  spaces with more general normed linear spaces,  and also replace the  uniform requirement for all $t\in \mathbb R$ with the relaxed $t\leq t_0$   because the former   is  rarely satisfied for the the nonlinear hyperbolic models  in  time-dependent phase spaces.

 (iii) The results of application  can be seen as an extension of those in \cite{pata} because under the same assumptions as in \cite{pata}, we show not only the upper semicontinuity  and residual continuity of the pullback $\mathscr D$-attractor of model \eqref{06231} w.r.t.   perturbation parameter $\e$,  but also the existence and the regularity of the  pullback $\mathscr D$-exponential attractors and  their H\"{o}lder continuity w.r.t.   perturbation parameter $\e\in (0, 1]$, which implies that the fractal dimension of the sections of the invariant time-dependent global attractor as shown  in \cite{pata} are uniformly bounded. The method developed here allows to  overcome  the difficulty of the hyperbolicity of the model.

 The paper is organized as follows. In Section 2, we establish two abstract criteria on the upper semicontinuity and the residual continuity of the pullback $\mathscr D$-attractor, respectively. In Section 3, we give the definition of the pullback $\mathscr D$-exponential attractors and discuss their existence and H\"{o}lder continuity criterion  at an abstract level. In Section 4, we apply the above mentioned criteria to model   \eqref{06231} to show  the continuity of the related  pullback $\mathscr D$-attractors,  and the existence and the H\"{o}lder continuity of the related pullback $\mathscr D$-exponential attractors w.r.t.   the perturbation  parameter $\e\in (0, 1]$.

\section{Continuity of pullback $\mathscr D$-attractor}
In this section, we  discuss the continuity of pullback $\mathscr D$-attractor $\mathcal A_\lambda=\{A_\lambda (t)\}_{t\in\mathbb{R}}$ w.r.t.  parameter $\lambda\in \Lambda$. For clarity, we first quote some notations, which will be used in the following sections.

Let $\{(X_t, d_t)\}_{t\in\mathbb{R}}$ be a family of metric spaces,
\begin{equation*}
  \mathbb{B}_t(x_0;R):=\{x\in X_t\ |\ d_t(x, x_0)\leq R\}
\end{equation*}
denote   the $R$-ball of $X_t$  centered at $x_0\in X_t$, especially,  $\mathbb{B}_t(R)= \mathbb{B}_t(0;R)$, and
\begin{equation*}
  \o^\e_t(B):=\bigcup_{x\in B}\o^\e_t(x):=\bigcup_{x\in B}\{y\in X_t\ |\ d_t(y, x)<\e\}
\end{equation*}
denote   the $\e$-neighborhood of a set $B\subset X_t$, $[B]_{X_t}$ denote   the closure of $B$ in $X_t$ and the letters $\mathcal A,\mathcal B, \cdots \mathcal E$ denote the time-dependent families, respectively,  for example
\begin{equation*}
  \mathcal D=\{D(t)\}_{t\in\mathbb{R}} \ \ \text{with}\ \ D(t)\subset X_t, \ \ \forall t\in\mathbb R.
\end{equation*}

We denote   the symmetric Hausdorff distance of two nonempty sets $B, C\subset X_t$ by
\begin{equation*}
  \mathrm{dist}^{symm}_{X_t}(B, C):=\max\{\mathrm{dist}_{X_t}(B, C), \mathrm{dist}_{X_t}(C, B)\}.
  \end{equation*}
 Then $(CB(X_t),  \mathrm{dist}^{symm}_{X_t})$ is a complete metric space, where $CB(X_t)$ is the collection of all nonempty
closed and  bounded subsets of $X_t$.

\begin{assumption}\label{3.1'}
Let $(\Lambda, \rho(\cdot, \cdot))$ be a complete metric space,  and $\{U_\lambda(t, \tau)\}_{\lambda\in\Lambda}$ be a family of parameterized  processes acting  on time-dependent metric spaces $(X_t, d_t)$.  Assume  that
\begin{description}
  \item[$(L_1)$]  $U_\lambda(t, \tau)$ has a pullback $\mathscr D$-attractor $\mathcal A_\lambda=\{A_\lambda (t)\}_{t\in\mathbb{R}}$ for each $\lambda\in \Lambda$;
  \item[$(L_2)$] there exist a  family   $\mathcal B=\{B(t)\}_{t\in\mathbb{R}}$, with $\emptyset\neq B(t)\subset X_t$, and a $t_0\in\mathbb{R}$  such that
  \begin{equation}\label{3.1}
  C(t):=\bigcup_{\lambda\in \Lambda}A_\lambda(t)\subset B(t), \ \ \forall t\leq t_0,
\end{equation}
  and
  $\mathcal A _\lambda$ pullback attracts $\mathcal B$, that is,
   \begin{equation*}
   \lim_{\tau\rightarrow -\infty} \mathrm{dist}_{X_t}\left(U_\lambda(t, \tau) B(\tau), A_\lambda(t)\right)=0, \ \ \ \forall t\in \mathbb R, \ \lambda\in \Lambda;
    \end{equation*}
      \item[$(L_3)$] for every $t\in \mathbb R$ and $s\leq \min\{t, t_0\}$,
  \begin{equation*}
    \lim_{\lambda\rightarrow \lambda_0}\sup_{x\in B(s)}d_t\left(U_\lambda(t, s)x,  U_{\lambda_0}(t, s)x\right)=0,  \ \ \forall \lambda_0\in\Lambda.
  \end{equation*}
\end{description}
\end{assumption}

\subsection{Upper semicontinuity of pullback $\mathscr D$-attractor}
We   show a criterion on the upper semicontinuity of pullback $\mathscr D$-attractor in time-dependent phase space in this subsection. For the related criterion on that in fixed metric space, i.e.,   $(X_t, d_t)\equiv(X, d), \forall t\in \mathbb R$, one can see  \cite{Carvalho1}.

\begin{theorem}\label{31} Let Assumption \ref{3.1'} hold. Then  the family of pullback $\mathscr D$-attractors $\mathcal A_\lambda=\{A_\lambda(t)\}_{t\in\mathbb{R}}$ is upper semicontinuous at each point $\lambda_0\in \Lambda$, that is,
\begin{equation}\label{3.0}
  \lim_{\lambda\rightarrow \lambda_0}\mathrm{dist}_{X_t}\left(A_\lambda(t), A_{\lambda_0}(t)\right)=0, \ \ \forall t\in\mathbb{R}.
\end{equation}
Moreover,   if the metric space $\Lambda$ is compact, then the family $\mathcal C=\{C(t)\}_{t\in \mathbb R}$ given by \eqref{3.1} possesses the following properties:
\begin{enumerate}[(i)]
  \item $C(t)$ is compact in $X_t$ for each  $t\in\mathbb{R}$;
  \item $\mathcal A_\lambda$ pullback attracts the family $\mathcal C$ for each $\lambda\in \Lambda$, that is,
      \begin{equation*}
      \lim_{\tau\rightarrow -\infty} \mathrm{dist}_{X_t} \left(U_\lambda(t, \tau) C(\tau), A_\lambda(t)\right)=0, \ \ \forall t\in \mathbb R.
      \end{equation*}
\end{enumerate}
\end{theorem}

\begin{proof}  It follows from condition $(L_2)$ that for any   $\lambda_0\in \Lambda$, $t\in\mathbb{R} $ and $\e>0$,  there exists a $T>0$ such that $t-T\leq t_0$ and
\begin{equation*}
  \mathrm{dist}_{X_t}\left(U_{\lambda_0}(t, t-T)B(t-T), A_{\lambda_0}(t)\right)<\frac{\e}{2},
\end{equation*}
which means
\begin{equation}\label{3.2}
  U_{\lambda_0}(t, t-T)B(t-T)\subset \o_t^{\e/2}\left(A_{\lambda_0}(t)\right).
\end{equation}
 By  condition $(L_2)$ and the invariance of the pullback $\mathscr D$-attractor,
\begin{equation}\label{3.3}
 A_\lambda(t)=U_\lambda(t, t-T)A_\lambda(t-T)\subset U_{\lambda}(t, t-T)B(t-T), \ \ \forall \lambda\in \Lambda.
\end{equation}
By condition $(L_3)$,
\begin{equation*}
    \lim_{\lambda\rightarrow \lambda_0}\sup_{x\in B(t-T)}d_t\left(U_\lambda(t, t-T)x,  U_{\lambda_0}(t, t-T)x\right)=0,
  \end{equation*}
 which means that    for the given $\e>0$, there exists a  $\delta=\delta(\e)>0$ such that
  \begin{equation}\label{3.4}
 U_{\lambda}(t, t-T)B(t-T)\subset \o_t^{\e/2}\left( U_{\lambda_0}(t, t-T)B(t-T)\right)\ \ \hbox{as} \ \ \rho(\lambda, \lambda_0)<\delta.
  \end{equation}
  The combination of \eqref{3.2}-\eqref{3.4} shows that
  \begin{equation*}
    A_\lambda(t)\subset \o_t^{\e/2}\left( U_{\lambda_0}(t, t-T)B(t-T)\right)\subset \o_t^{\e}\left(A_{\lambda_0}(t)\right)\ \ \hbox{as} \ \ \rho(\lambda, \lambda_0)<\delta.
  \end{equation*}
    That is, $\mathrm{dist}_{X_t}\left(A_\lambda(t), A_{\lambda_0}(t)\right)<\e$. Hence,  the upper semicontinuity   \eqref{3.0} holds.

 Moreover,  for  any sequence $\{x_n\} \subset C(t)=\cup_{\lambda\in \Lambda}A_\lambda(t)\subset B(t)$,   there exists  a sequence $\{\lambda_n\}\subset \Lambda$ such that $x_n\in A_{\lambda_n}(t)$. Taking into account  the compactness of $\Lambda$,  we have (subsequence if necessary) $\lambda_n\rightarrow \lambda_0$ in $\Lambda$. By the compactness of $A_{\lambda_0}(t)$ in $X_t$ and formula \eqref{3.0}, there exists a   sequence $\{y_n\}\subset A_{\lambda_0}(t)$ such that
  \begin{equation*}
    \lim_{n\rightarrow \infty} d_t(x_n, y_n)=\lim_{n\rightarrow \infty}\mathrm{dist}_{X_t}\left(x_n, A_{\lambda_0}(t)\right)\leq \lim_{n\rightarrow \infty}\mathrm{dist}_{X_t}\left(A_{\lambda_n}(t), A_{\lambda_0}(t)\right)=0,
  \end{equation*}
  and there exists  a subsequence $\{y_{n_k}\}\subset \{y_n\}$ such that
  \begin{equation*}
    y_{n_k}\rightarrow x_0\in A_{\lambda_0}(t)\ \ \hbox{in}\ \ X_t.
  \end{equation*}
  Therefore,
  \begin{equation*}
   \lim_{k\rightarrow \infty} d_t\left(x_{n_k}, x_0\right)\leq  \lim_{k\rightarrow \infty} d_t\left(x_{n_k}, y_{n_k}\right)+ \lim_{k\rightarrow \infty} d_t\left(y_{n_k}, x_0\right)=0,
  \end{equation*}
which means that the set  $C(t)$ is  compact  in $X_t$ for each  $t\in\mathbb{R}$.

 Taking into account the fact  that   $C(t)\subset B(t)$ for all $t\leq t_0$ and  that $\mathcal A_\lambda$ pullback attracts the family $\mathcal B$, we have that $\mathcal A_\lambda$ pullback attracts the family $\mathcal C$ for each  $\lambda\in \Lambda$.
  \end{proof}

\begin{remark}\label{rem1}    Theorem \ref{31} implies that  when $\Lambda$ is a compact metric space,  we can take the family $\mathcal B=\{B(t)\}_{t\in \mathbb R}=\mathcal C=\{C(t)\}_{t\in \mathbb R}$ in  conditions $(L_2)-(L_3)$, in other words, conditions $(L_2)-(L_3)$ hold on $\mathcal C$.
\end{remark}

\subsection{Residual continuity of pullback $\mathscr D$-attractor}

In this subsection,  we show a residual continuity  criterion on the pullback $\mathscr D$-attractor in time-dependent phase spaces.  For the  related  criteria on  the residual continuity of global, uniform and pullback attractors in fixed metric space,  one can see  \cite{Babin1, Hoang1, Hoang2}. For clarity, we first quote a definition of the residual subset.

\begin{definition}\label{Residual}
A set is said to be nowhere dense if its closure contains no nonempty open sets.  A set  is said to be a \textit{residual} set if its complement is a countable union of  nowhere dense sets.
\end{definition}

Any residual subset of the complete metric space $\Lambda$ is dense in $\Lambda$.

\begin{theorem}\label{32} Let Assumption \ref{3.1'} be vaild and  $\{U_\lambda(t, \tau)\}_{\lambda\in \Lambda}$ be a family of continuous processes, i.e., $U_\lambda(t,\tau):X_\tau\rightarrow X_t$ is a continuous  operator for each $t\geq \tau$ and $\lambda\in \Lambda$. Assume that either (i) the set $B(t)$ is compact in $X_t$ for all $t\leq t_0$  or (ii) the metric space  $\Lambda$ is compact.
   Then  there exists a residual subset $\Lambda^*$ of $\Lambda$ such that the pullback $\mathscr D$-attractor $\mathcal A_\lambda=\{A_\lambda(t)\}_{t\in\mathbb{R}}$ is continuous at each point $\lambda_0\in \Lambda^*$, i.e.,
  \begin{equation}\label{3.5}
    \lim_{\lambda\rightarrow \lambda_0}\mathrm{dist}_{X_t}^{symm}\left(A_\lambda(t), A_{\lambda_0}(t)\right)=0,\ \ \forall  t\in\mathbb{R}.
  \end{equation}
\end{theorem}

By Remark \ref{rem1}, when $\Lambda$ is a compact metric space,  the conditions $(L_2)-(L_3)$ of Assumption \ref{3.1'} hold on  the family $\mathcal C=\{C(t)\}_{t\in \mathbb R}$ as shown in   \eqref{3.1}, and Theorem  \ref{31}
shows that each section $C(t)$    is compact in $X_t$.
 Thus, it is enough  to prove Theorem \ref{32} in  case (i) for one only needs to replace $\mathcal B$ there  by  $\mathcal C$ in case (ii). In order to prove Theorem \ref{32}, we  need the following lemmas.

\begin{lemma}\label{33}\cite{Hoang1} Let $f_n:X\rightarrow Y$ be a continuous map for each $n\in\mathbb{N}$, where $X$ is a complete metric space and $Y$ is a metric space. Assume that $f$ is the pointwise limit of $f_n$, i.e.,
\begin{equation*}
  f(x)=\lim_{n\rightarrow \infty}f_n(x), \ \ \forall x\in X.
\end{equation*}
Then, there exists a residual subset $\mathcal R$ of $X$ such that $f$ is continuous at every point  $x\in \mathcal R$.
 \end{lemma}

\begin{lemma}\label{34}  Let the assumptions of Theorem \ref{32} be valid. Then  the mapping $\lambda \mapsto  [U_\lambda(t, s)B(s)]_{X_t}$ is continuous from $\Lambda$ into $CB(X_t)$  for all $t\in\mathbb{R}$ and $s\leq \min\{t, t_0\}$.
\end{lemma}

\begin{proof} For any   $\lambda\in \Lambda$, $t\in\mathbb R$ and $s\leq \min\{t, t_0\}$, it follows  from the continuity of operator $U_\lambda(t, s): X_s \rightarrow X_t$ and the compactness of $B(s)$ in $X_s$ that
\begin{equation*}
U_\lambda(t, s)B(s)=\big[U_\lambda(t, s)B(s) \big]_{X_t}\in CB(X_t).
\end{equation*}
 Thus, the mapping $\lambda \mapsto \left[U_\lambda(t, s)B(s)\right]_{X_t}$ is from $\Lambda$ into $CB(X_t)$. Repeating  the similar argument as Lemma 3.1 in \cite{Hoang1,Hoang2}, one   easily obtains  the continuity of this mapping.   We omit the details  here.
\end{proof}

\begin{lemma}\label{35}   Let the assumptions of Theorem \ref{32} be valid. Then  the mapping $(\lambda, B)\mapsto U_\lambda(t, s) B$ is continuous at every point  $(\lambda, B)\in \Lambda\times CB(B(s))$ for all $t\in \mathbb R$ and $s\leq \min\{t, t_0\}$.
\end{lemma}
\begin{proof}  It follows from condition $(L_3)$ that for any  $(\lambda_0, B_0)\in \Lambda\times CB(B(s))$ and $\e>0$, there exists a $\delta=\delta(\e)>0$ such that
\begin{equation}\label{3.6}
\sup_{x\in B(s)}d_t\left(U_\lambda(t, s)x, U_{\lambda_0}(t, s)x\right)<\frac{\e}{2}\ \ \hbox{as}\ \   \rho(\lambda, \lambda_0)<\delta.
\end{equation}
Since $B(s)$ is compact in $X_s$ and the operator $U_\lambda(t, s): X_s\rightarrow X_t$ is continuous, $U_\lambda(t, s)$ is uniformly continuous on $B(s)$ and $U_\lambda(t, s)B$ is compact in $X_t$ for all $(\lambda, B)\in \Lambda\times CB(B(s))$. Thus  there exists a constant $\delta_1:0<\delta_1<\delta$ such that for every $x, y\in B(s)$ with $d_s(x, y)<\delta_1$,
\begin{equation}\label{3.7}
  d_t\left(U_{\lambda_0}(t, s)x, U_{\lambda_0}(t, s)y\right)<\frac{\e}{2}.
\end{equation}
For any  point $(\lambda, B)\in \Lambda\times CB(B(s))$ with
\begin{equation}\label{c1}
\rho(\lambda, \lambda_0)<\delta_1\ \ \hbox{and}\ \ \mathrm{dist}^{symm}_{X_s}(B, B_0)<\delta_1,
\end{equation}
it follows  from  the compactness of $B_0$ in $X_s$ that  for every $b\in B$, there exists the best approximating element  $b_0\in B_0$ such that
\begin{equation}\label{3.8}
 d_s(b, b_0)=\mathrm{dist}_{X_s}(b, B_0)\leq \mathrm{dist}^{symm}_{X_s}(B, B_0)<\delta_1.
\end{equation}
 The combination of \eqref{3.6}-\eqref{3.8}  turns out
\begin{equation*}
  \begin{split}
      &\mathrm{dist}_{X_t}\left(U_\lambda(t, s)b, U_{\lambda_0}(t, s)B_0\right)\\
      \leq \  & d_t \left(U_\lambda(t, s)b,  U_{\lambda_0}(t, s)b_0\right)\\
    \leq\ &d_t \left(U_\lambda(t, s)b,  U_{\lambda_0}(t, s)b\right)+ d_t \left(U_{\lambda_0}(t, s)b, U_{\lambda_0}(t, s)b_0\right)\\
     <\ &\e/2+\e/2=\e.
  \end{split}
\end{equation*}
By the arbitrariness of $b\in B$,
\begin{equation*}
 \mathrm{dist}_{X_t}\left(U_\lambda(t, s)B, U_{\lambda_0}(t, s)B_0\right)\leq \e.
\end{equation*}
On the other hand, by the compactness of $B$ in $X_s$ and the symmetry of $\mathrm{dist}^{symm}_{X_s}(B, B_0)<\delta_1$, we also have
\begin{equation*}
\mathrm{dist}_{X_t}\left(U_\lambda(t, s)B_0, U_{\lambda_0}(t, s)B\right)\leq \e.
\end{equation*}
Therefore, for any $(\lambda, B)\in \Lambda\times CB(B(s))$, we arrive at
\begin{equation*}
 \mathrm{dist}^{symm}_{X_t}\left(U_\lambda(t, s)B, U_{\lambda_0}(t, s)B_0\right)\leq \e\ \ \hbox{as}\ \ \eqref{c1}\ \ \hbox{holds}.
\end{equation*}
  By the  arbitrariness of $(\lambda_0, B_0)\in \Lambda\times CB(B(s))$, we obtain the desired conclusion.
\end{proof}

\begin{proof}[\textbf{Proof of Theorem \ref{32}}] It follows from condition $(L_2)$  and the invariance of pullback $\mathscr D$-attractor that
\begin{equation*}
  A_\lambda(n)=U_\lambda(n, s)A_\lambda(s)\subset U_\lambda(n, s) B(s), \ \ \forall \lambda\in \Lambda,\ n\in \mathbb{Z}, \  s\leq \min\{n, t_0\},
\end{equation*}
which implies
\begin{equation}\label{3.9}
  \mathrm{dist}_{X_n}\left(A_\lambda(n), U_\lambda(n, s) B(s)\right)=0, \ \ \forall \lambda\in \Lambda,\ n\in \mathbb{Z}, \  s\leq \min\{n, t_0\}.
\end{equation}
While condition $(L_2)$ shows that
\begin{equation}\label{3.10}
  \lim_{s\rightarrow -\infty}\mathrm{dist}_{X_n}\left(U_\lambda(n, s)B(s), A_\lambda(n) \right)=0, \ \ \forall \lambda\in \Lambda,\ n\in \mathbb{Z}.
\end{equation}
The combination of \eqref{3.9}-\eqref{3.10} yields
\begin{equation*}
  \lim_{s\rightarrow -\infty}\mathrm{dist}^{symm}_{X_n}\left(U_\lambda(n, s) B(s), A_\lambda(n)\right)=\lim_{s\rightarrow -\infty}\mathrm{dist}_{X_n}\left(U_\lambda(n, s) B(s), A_\lambda(n)\right)=0,
\end{equation*}
that is,
\begin{equation}\label{3.11}
  A_\lambda(n)=\lim_{s\rightarrow -\infty} U_\lambda(n, s) B(s), \ \ \forall \lambda\in \Lambda,\ n\in \mathbb{Z}.
\end{equation}

For every $n\in \mathbb Z$ and $s\leq \min\{n, t_0\}$,  we define the mappings $f^n_s, f^n:\Lambda\rightarrow CB(X_n)$,
\begin{equation*}
   f^n_s(\lambda)=  U_\lambda(n, s) B(s), \ \ f^n(\lambda)= A_\lambda(n),  \ \  \forall \lambda\in \Lambda.
\end{equation*}
Then formula \eqref{3.11} reads $f^n(\lambda)=\lim_{s\rightarrow -\infty}f^n_s(\lambda)$.   Lemma \ref{34} shows that $f^n_s(\lambda)$ is continuous from $\Lambda$ into $CB(X_n)$ for each $n\in \mathbb{Z}$ and $s\leq \min\{n, t_0\}$. Therefore,  by Lemma \ref{33},  there exists a residual subset $\Lambda_n\subset \Lambda$ for each $n\in \mathbb{Z}$ such that the limiting function $f^n$ is continuous on $\Lambda_n$, that is,
 \begin{equation}\label{01061}
  \lim_{\lambda\rightarrow \lambda_0}\mathrm{dist}^{symm}_{X_n}\left(f^n(\lambda), f^n(\lambda_0)\right)=0,\ \ \forall \lambda_0\in \Lambda_n, \   n\in \mathbb{Z}.
\end{equation}
Let
\begin{equation*}
  \Lambda^*=\bigcap_{n\in\mathbb{Z}} \Lambda_n.
\end{equation*}
   Obviously,  $\Lambda^*$ is a residual subset of $\Lambda$ for the countable intersection of residual subsets is also a residual subset, and by \eqref{01061}
\begin{equation}\label{3.12}
  \lim_{\lambda\rightarrow \lambda_0}\mathrm{dist}^{symm}_{X_n}\left(A_\lambda(n), A_{\lambda_0}(n)\right)=0,\ \ \forall \lambda_0\in \Lambda^*, \ \ \forall n\in \mathbb{Z}.
\end{equation}
Take $n\leq \min\{t, t_0\}$ for any given $t\in\mathbb{R}$. By the invariance of $\mathcal A_\lambda$, we have
\begin{equation}\label{3.13}
  A_\lambda(t)=U_\lambda(t, n)A_\lambda(n), \ \ \forall \lambda\in \Lambda.
\end{equation}
Taking into account  $A_{\lambda_0}(n)\in CB(B(n)), \forall\lambda_0\in \Lambda^*$, we infer from
Lemma \ref{35} that for any $\e>0$, there exists a $\delta>0$ such that
\begin{equation}\label{3.14}
 \mathrm{dist}^{symm}_{X_t}\left(U_\lambda(t, n)B, U_{\lambda_0}(t, n)A_{\lambda_0}(n)\right)<\e
\end{equation}
 whenever $(\lambda, B)\in\Lambda\times CB(B(n))$ and $\rho(\lambda, \lambda_0)+ \mathrm{dist}^{symm}_{X_n}\left(B, A_{\lambda_0}(n)\right)<\delta$.
 Formula \eqref{3.12} means that  there exists a constant  $\delta_1: 0<\delta_1<\delta/2$  such that
\begin{equation}\label{3.15}
 \mathrm{dist}^{symm}_{X_n}\left(A_\lambda(n), A_{\lambda_0}(n)\right)< \delta/2 \ \ \hbox{as}\ \  \rho(\lambda, \lambda_0)<\delta_1.
\end{equation}
The combination  of \eqref{3.13}-\eqref{3.15}  arrives at
\begin{equation*}
 \mathrm{dist}^{symm}_{X_t}\left(A_\lambda(t), A_{\lambda_0}(t)\right)= \mathrm{dist}^{symm}_{X_t}\left(U_\lambda(t, n)A_\lambda(n), U_{\lambda_0}(t, n)A_{\lambda_0}(n)\right)<\e
\end{equation*}
whenever $\lambda\in \Lambda$ and $\rho(\lambda, \lambda_0)<\delta_1$ for  $A_\lambda(n)\in CB(B(n))$.  That is,  the mapping  $\lambda \mapsto A_\lambda(t)$ is continuous on $\Lambda^*$ for all $t\in \mathbb{R}$.
\end{proof}

\subsection{The equivalence between the  continuity and the equi-attraction}
In  this  section, motivated by the idea in \cite{Hoang1,Hoang2}, we show that the continuity of pullback $\mathscr D$-attractor w.r.t. the parameter $\lambda\in \Lambda$ (not only the residual subset $\Lambda^*$)  is equivalent to the pullback equi-attraction under some assumptions. For the related equivalence criterion  on that of    global, pullback and uniform attractors, one can see \cite{Hoang1,Hoang2}, and for a comprehensive summary  on this topic, one can see   \cite{Kloeden2, Kloeden3}.

\begin{theorem}\label{36} Let the assumptions of Theorem \ref{32} be valid.
\begin{enumerate}[$(i)$]
  \item If the family $\mathcal B=\{B(t)\}_{t\in \mathbb R}$ as shown in condition $(L_2)$ satisfies the  pullback equi-attraction at time $t\in \mathbb R$, that is,
  \begin{equation}\label{3.16}
    \lim_{s\rightarrow -\infty}\sup_{\lambda\in \Lambda}\mathrm{dist}_{X_t}\left(U_\lambda(t, s)B(s), A_\lambda(t)\right)=0,
  \end{equation}
 then the mapping  $\lambda\mapsto A_\lambda(t)$ is continuous on $\Lambda$, that is,
  \begin{equation*}
    \lim_{\lambda\rightarrow \lambda_0}\mathrm{dist}^{symm}_{X_t}\left(A_\lambda(t), A_{\lambda_0}(t)\right)=0, \ \ \forall \lambda_0\in \Lambda.
  \end{equation*}
  \item   Assume that $\Lambda$ is  compact  and  there is a function $\gamma(t)$, with $\gamma(t)\leq \min\{t, t_0\}$,  such that
   \begin{equation}\label{3.161}
    \bigcup_{\lambda\in \Lambda} U_\lambda(t, s)B(s)\subset B(t),\ \  \forall s\leq \gamma(t).
   \end{equation}
  If the mapping $\lambda\mapsto A_\lambda(t)$ is continuous on $\Lambda$, then  the pullback equi-attraction \eqref{3.16} holds.
\end{enumerate}
\end{theorem}

In order to prove Theorem \ref{36}, we need the following Dini's theorem.

\begin{lemma}\label{37}  (Theorem 4.1 in \cite{Hoang1}) Let $K$ be a compact metric space,  $Y$ be a metric space, and $f_n: K\rightarrow Y$ be a continuous mapping  for each $n\in \mathbb{N}$. Assume that $f_n$ converges to a continuous function $f: K\rightarrow Y$ as $n\rightarrow \infty$ in the following monotonic way
\begin{equation*}
  d_Y\left(f_{n+1}(x), f(x)\right)\leq d_Y\left(f_n(x), f(x)\right), \ \ \forall x\in K, \ \ \forall n\geq 1.
\end{equation*}
Then $f_n$ converges to $f$ uniformly on $K$ as $n\rightarrow \infty$.
\end{lemma}

\begin{proof}[\textbf{Proof of Theorem \ref{36}}] (i) By condition $(L_2)$ and the invariance of pullback $\mathscr D$-attractor,
\begin{equation*}
  A_\lambda(t)=U_\lambda(t, s)A_\lambda(s)\subset U_\lambda(t, s)B(s),\ \ \forall s\leq \min\{t, t_0\},   t\in \mathbb{R}, \lambda\in \Lambda,
\end{equation*}
which implies
\begin{equation}\label{3.17}
  \sup_{\lambda\in \Lambda}\mathrm{dist}_{X_t}\left(A_\lambda(t), U_\lambda(t, s)B(s)\right)=0, \ \ \forall s\leq \min\{t, t_0\}.
\end{equation}
The combination of   \eqref{3.16} and \eqref{3.17} turns out
\begin{equation*}
    \lim_{s\rightarrow -\infty}\sup_{\lambda\in \Lambda}\mathrm{dist}^{symm}_{X_t}\left(A_\lambda(t), U_\lambda(t, s)B(s)\right) = \lim_{s\rightarrow -\infty}\sup_{\lambda\in \Lambda}\mathrm{dist}_{X_t}\left(U_\lambda(t, s)B(s), A_\lambda(t)\right)=0,
\end{equation*}
which means that for any   $\e>0$, there exists  a $s_0\leq \min\{t, t_0\}$ such that
\begin{equation}\label{3.18}
\sup_{\lambda\in \Lambda}\mathrm{dist}^{symm}_{X_t}\left(A_\lambda(t), U_\lambda(t, s_0)B(s_0)\right)<\frac{\e}{3}.
\end{equation}
  Lemma \ref{34} shows that  for any   $\lambda_0\in \Lambda$, there exists a $\delta>0$ such that
\begin{equation}\label{3.19}
\mathrm{dist}^{symm}_{X_t}\left(U_\lambda(t, s_0)B(s_0), U_{\lambda_0}(t, s_0)B(s_0)\right) <\frac{\e}{3}  \ \ \hbox{as}\ \   \rho(\lambda, \lambda_0)<\delta.
 \end{equation}
 The combination of \eqref{3.18}-\eqref{3.19} yields
 \begin{equation*}
   \begin{split}
   \mathrm{dist}^{symm}_{X_t}\left(A_\lambda(t), A_{\lambda_0}(t)\right)
   \leq &  \mathrm{dist}^{symm}_{X_t}\left(A_\lambda(t), U_\lambda(t, s_0)B(s_0)\right)\\
       & + \mathrm{dist}^{symm}_{X_t}\left(U_\lambda(t, s_0)B(s_0), U_{\lambda_0}(t, s_0)B(s_0)\right)\\
       & + \mathrm{dist}^{symm}_{X_t}\left(U_{\lambda_0}(t, s_0)B(s_0), A_{\lambda_0}(t)\right)\\
        < & \frac{\e}{3}+\frac{\e}{3}+\frac{\e}{3}=\e\ \ \hbox{as} \ \ \rho(\lambda, \lambda_0)<\delta,
   \end{split}
 \end{equation*}
that is, the mapping $\lambda \mapsto A_\lambda(t)$  is continuous at $\lambda_0$. By the arbitrariness of $\lambda_0\in \Lambda$ we obtain  that the mapping $\lambda\mapsto A_\lambda(t)$ is continuous on $\Lambda$.
 \bigskip

 (ii) If formula \eqref{3.161} holds,  then  for any given  $t\in \mathbb{R}$,   we take $s_0= \gamma(t)-1\leq \min\{t, t_0\}-1$ and
 \begin{equation*}
  s_n=\gamma(s_{n-1})-1\leq \min\{s_{n-1}, t_0\}-1, \ \ \forall n\geq 1.
 \end{equation*}
 Obviously, the sequence $\{s_n\}_{n=0}^\infty$ is strictly decreasing and $\lim_{n\rightarrow \infty}s_n= -\infty$. By formula \eqref{3.161},
 \begin{equation*}
   U_\lambda(t, s_{n+1})B(s_{n+1})=U_\lambda(t, s_n)U_\lambda(s_n, s_{n+1}) B(s_{n+1})\subset U_\lambda(t, s_n)B(s_n),
 \end{equation*}
and hence,
 \begin{equation}\label{3.20}
 \mathrm{dist}_{X_t}\left(U_\lambda (t, s_{n+1})B(s_{n+1}), A_\lambda(t)\right)\leq \mathrm{dist}_{X_t}\left(U_\lambda (t, s_n)B(s_n), A_\lambda(t)\right), \ \ \forall \lambda\in \Lambda,\ n\geq 1.
 \end{equation}
 The combination of \eqref{3.17} and \eqref{3.20} shows that
\begin{equation}\label{3.22}
 \mathrm{dist}^{symm}_{X_t}\left(U_\lambda (t, s_{n+1})B(s_{n+1}), A_\lambda(t)\right)\leq \mathrm{dist}^{symm}_{X_t}\left(U_\lambda (t, s_n)B(s_n), A_\lambda(t)\right), \ \ \forall \lambda\in \Lambda,\ n\geq 1.
 \end{equation}

 Let the mappings $f^t_n, f^t: \Lambda\rightarrow CB(X_t)$,
 \begin{equation*}
 f^t_n(\lambda)= U_\lambda(t, s_n)B(s_n), \ \ f^t(\lambda)=A_\lambda(t), \ \ \forall \lambda\in \Lambda, \ n\geq 1.
 \end{equation*}
  Lemma \ref{34} shows that  $f^t_n: \Lambda\rightarrow CB(X_t)$ is  continuous for each $n\geq 1$,  and formula \eqref{3.22} reads
 \begin{equation*}
   \mathrm{dist}^{symm}_{X_t}\left(f^t_{n+1}(\lambda), f^t(\lambda) \right)\leq \mathrm{dist}^{symm}_{X_t}\left(f^t_n(\lambda), f^t(\lambda)\right), \ \ \forall \lambda\in \Lambda, \ n\geq 1.
 \end{equation*}
 By the continuity of the mapping $f^t: \Lambda\rightarrow CB(X_t)$ and Lemma \ref{37}
 we know that $f^t_n$ converges to $f^t$ uniformly on $\Lambda$ as $n\rightarrow \infty$, that is,
 \begin{equation}\label{3.23}
   \lim_{n\rightarrow \infty}\sup_{\lambda\in \Lambda} \mathrm{dist}^{symm}_{X_t}\left(U_\lambda(t, s_n)B(s_n), A_\lambda(t)\right)=0.
 \end{equation}
 For any $s\in (s_{n+2}, s_{n+1})$, we have that $s<s_{n+1}<\gamma(s_n)\leq \min\{s_n, t_0\}$. Thus  it follows from formula \eqref{3.161} that
 \begin{equation*}
 U_\lambda(t, s)B(s)= U_\lambda(t, s_n)U_\lambda(s_n, s)B(s)\subset U_\lambda(t, s_n)B(s_n), \ \ \forall \lambda\in \Lambda,
 \end{equation*}
and hence,
 \begin{equation*}
 \begin{split}
   \mathrm{dist}_{X_t}\left(U_\lambda(t, s)B(s), A_\lambda(t)\right)&\leq \mathrm{dist}_{X_t}\left(U_\lambda(t, s_n)B(s_n), A_\lambda(t)\right)\\
   &\leq \mathrm{dist}^{symm}_{X_t}\left(U_\lambda(t, s_n)B(s_n), A_\lambda(t)\right), \ \ \forall \lambda\in \Lambda.
   \end{split}
 \end{equation*}
 Therefore, we infer from formula \eqref{3.23} that
 \begin{equation*}
   \lim_{s\rightarrow -\infty} \sup_{\lambda\in \Lambda}\mathrm{dist}_{X_t}\left(U_\lambda(t, s)B(s), A_\lambda(t)\right)\leq \lim_{n\rightarrow \infty}\sup_{\lambda\in \Lambda}\mathrm{dist}^{symm}_{X_t}\left(U_\lambda(t, s_n)B(s_n), A_\lambda(t)\right)=0.
 \end{equation*}
 This completes the proof.
\end{proof}

\section{Continuity of the pullback $\mathscr D$-exponential attractors}

The purpose of this section is to show an existence and continuity  criterion on the  pullback $\mathscr D$-exponential attractors, which can be seen as an extension of the related criterion  on the  pullback exponential attractors in fixed metric space  (cf. \cite{Y-LyPEA}).  We first give a proper notion of the pullback $\mathscr D$-exponential attractor.

\begin{definition}\label{PDEA} A family $\mathcal E=\{E(t)\}_{t\in\mathbb{R}}$ is called a pullback $\mathscr D$-exponential attractor of the process $U(t, \tau):X_\tau\rightarrow X_t$, if
\begin{enumerate}[(i)]
  \item  each section $E(t)$ is compact in $X_t$  and the fractal dimension of $E(t)$ in $X_t$ is uniformly bounded, that is,
  \begin{equation*}
    \sup_{t\in\mathbb{R}}\mathrm{dim}_f\left(E(t), X_t\right)<+\infty;
  \end{equation*}
  \item $\mathcal E$ is semi-invariant, that is, $U(t, \tau)E(\tau)\subset E(t)$ for all $t\geq \tau$;
  \item there exists a positive constant $\beta$ such that
   \begin{equation*}
     \mathrm{dist}_{X_t}\left(U(t, t-\tau)D(t-\tau), E(t)\right)\leq C(\mathcal D, t)e^{-\beta\tau}, \ \ \forall \tau\geq \tau_0(\mathcal D, t),
   \end{equation*}
   for all $t\in \mathbb R$ and $\mathcal D\in \mathscr D$, where $C(\mathcal D, t)$ and $\tau_0(\mathcal D, t)$  are  positive constants   depending only  on $\mathcal D$ and $t$.
\end{enumerate}
\end{definition}

\begin{assumption}\label{assumptionea}
Let $(\Lambda, \rho(\cdot, \cdot))$ be a complete metric space and $\{U_\lambda(t, \tau)\}_{\lambda\in\Lambda}$ be a parameterized family of processes acting on time-dependent normed linear  spaces $\{X_t\}_{t\in \mathbb R}$.  Suppose that
\begin{description}
  \item[$(H_1)$] There exist  a time-dependent family $\mathcal B=\{B(t)\}_{t\in \mathbb R}$ and positive constants $T$, $R_0$ such that
      \begin{align}
        & B(t) \ \ \hbox{is closed in}\ \ X_t \ \  \hbox{and}\ \ B(t)\subset \mathbb B_t(R_0),\ \ \forall t\in \mathbb R, \nonumber\\
         & \cup_{\lambda\in \Lambda}U_\lambda(t, t-\tau)B(t-\tau)\subset B(t), \ \ \forall \tau\geq T, t\in \mathbb R, \label{14.1}
      \end{align}
      hereafter, $\mathbb B_t(R_0)=\mathbb B_t(0; R_0)$ is the $R_0$-ball of $X_t$    centered at  $0$.
  \item[$(H_2)$] \textit{(Quasi-stability)} There exist  a Banach space $Z$ with the compact seminorm $n_Z(\cdot)$ and a $t_0\in \mathbb R$  such that
      \begin{equation}\label{14.2}
        \|U_\lambda(t, t-T)x-U_\lambda(t, t-T)y\|_{X_t}\leq \eta \|x-y\|_{X_{t-T}}+ n_Z\left(K^\lambda_t x- K^\lambda_t y\right)
      \end{equation}
      for all $x, y\in B(t- T)$, $\lambda\in \Lambda$ and $t\leq t_0$, where $\eta\in (0, 1/2)$ and the mapping $K^\lambda_t: B(t-T)\rightarrow Z$ is uniformly Lipschitz continuous, that is,
      \begin{equation}\label{14.3}
        \sup_{\lambda\in \Lambda}\|K^\lambda_t x- K^\lambda_t y\|_Z\leq L\|x-y\|_{X_{t-T}}, \ \ \forall x, y\in B(t-T), t\leq t_0,
      \end{equation}
      for some  constant $L>0$.
  \item[$(H_3)$] \textit{(Lipschitz continuity)} There exists a uniform Lipschitz constant $L_1\geq 2$ such that
  \begin{equation}\label{14.4}
  \sup_{\lambda\in \Lambda}\|U_\lambda(t, t-\tau)x-U_\lambda(t, t-\tau)y\|_{X_t}\leq L_1 \|x-y\|_{X_{t-\tau}},
  \end{equation}
  for all $x, y\in B(t-\tau)$, $\tau\in [0, T]$ and $t\in \mathbb R$.
\end{description}
\end{assumption}

For any given $\lambda_0\in \Lambda$, we define the function
\begin{equation}\label{14.5}
\Gamma(\lambda, \lambda_0):= \sup_{t\leq t_0}\sup_{s\in [0, T]} \sup_{x\in B(t-s)}\|U_\lambda(t, t-s)x- U_{\lambda_0}(t, t-s)x\|_{X_t}, \ \ \forall \lambda\in \Lambda,
\end{equation}
and the set
$$\Lambda_0:=\{\lambda\in \Lambda\ |\  0<\Gamma(\lambda, \lambda_0)<1\}.$$

\begin{theorem}\label{42} Let Assumption \ref{assumptionea} be valid.  Then  for any given  $\lambda_0\in \Lambda$, there is a  semi-invariant family $\mathcal E_\lambda=\{E_\lambda(t)\}_{t\in \mathbb R}$, with $ \lambda\in \Lambda$, possessing the following properties:
\begin{enumerate}[(i)]
  \item each section $E_\lambda(t)\subset B(t)$ is compact in $X_t$ and its fractal dimension in $X_t$  is uniformly bounded, that is,
      \begin{equation}\label{14.6}
  \sup_{\lambda\in \Lambda}\sup_{t\in \mathbb{R}} \mathrm{dim}_f\left(E_\lambda(t), X_t\right)\leq \Big[\ln\big(\frac{1}{2\eta}\big)\Big]^{-1}\ln m_Z\Big(\frac{2L}{\eta}\Big)<\infty,
\end{equation}
where $m_Z(R)$ is the maximal number of elements $z_i$ in the ball $\{z\in Z|\|z\|_Z\leq R\}$ such that $n_Z(z_i-z_j)>1$, $i \neq j$;
  \item $\mathcal E_\lambda$ pullback attracts the family $\mathcal B$ at an exponential rate, that is,
      \begin{equation}\label{14.7}
      \mathrm{dist}_{X_t}\left(U_\lambda(t, t-\tau)B(t-\tau), E_\lambda(t)\right)\leq C(t)e^{-\beta \tau}, \ \ \forall \tau\geq \tau_t, \ t\in \mathbb R,
      \end{equation}
      where  $\beta$, $C(t)$ and $\tau_t$ are  positive constants independent of  $\lambda\in \Lambda$;
  \item $\mathcal E_\lambda$ is continuous at the point $\lambda_0$ in the following sense,
  \begin{equation*}
  \mathrm{dist}_{X_t}^{symm} \left( E_\lambda(t), E_{\lambda_0}(t)\right)\leq C_1(t)\Gamma_t(\lambda, \lambda_0), \ \ \forall t\in \mathbb R, \lambda\in \Lambda_0,
  \end{equation*}
  where  $C_1(t)$ is a positive constant depending only on $t$ and
  \begin{equation*}
\Gamma_t(\lambda, \lambda_0)={ \left\{
    \begin{array}{ll}
     \big[\Gamma(\lambda, \lambda_0)\big]^\gamma , & t\leq t_0, \\
~\\
      \big[\Gamma(\lambda, \lambda_0)\big]^\gamma +\sup_{x\in B(n_* T)}\|U_\lambda(t, n_* T)x- U_{\lambda_0}(t, n_* T)x\|_{X_t} , & t>t_0,
    \end{array}
  \right.}
\end{equation*}
with $\gamma\in (0, 1)$, $n_*\in \mathbb Z$ such that $t_0- n_* T\in [0, T)$.
\end{enumerate}
\end{theorem}

However, we can not ensure that the family $\mathcal E_\lambda=\{E_\lambda(t)\}_{t\in \mathbb R}$ as shown in Theorem \ref{42} is the pullback $\mathscr D$-exponential attractor of the process $U_\lambda(t, \tau)$ though it possesses the semi-invariance, the compactness and the boundedness of the fractal dimension because it only pullback attracts the  family $\mathcal B$  (rather than every family $\mathcal D\in \mathscr D$) at an exponential rate.
In order to guarantee that  the family $\mathcal E_\lambda$ is exactly  the desired  pullback $\mathscr D$-exponential attractor, we give the following    two  corollaries with additional assumptions.  We first introduce the definition of pullback $\mathscr D$-absorbing family which will be used here.

\begin{definition}\label{Absorbing}\textit{(Pullback $\mathscr D$-absorbing family)} A family  of nonempty sets $\mathcal B=\{B(t)\}_{t\in\mathbb R}$ is called a pullback $\mathscr D$-absorbing family of the process $U(t, \tau): X_\tau \rightarrow X_t$, if for any $t\in \mathbb R$ and any $\mathcal D\in \mathscr D$, there exists a $\tau_0(t, \mathcal D)\leq t$ such that
\begin{equation*}
U(t, \tau)D(\tau)\subset B(t)\ \ \hbox{for}\ \ \tau\leq  \tau_0(t, \mathcal D).
\end{equation*}
In particular, if for any $\mathcal D\in \mathscr D$, there exists a constant $e(\mathcal D)>0$ such that
\begin{equation*}
U(t, \tau)D(\tau)\subset B(t),\ \ \forall    \tau\leq  t-e(\mathcal D), \ t\in \mathbb R,
\end{equation*}
then $\mathcal B$ is called a  uniformly pullback $\mathscr D$-absorbing family.
\end{definition}

\begin{corollary}\label{ea1} Let   Assumption \ref{assumptionea} be valid, and   the family $\mathcal B=\{B(t)\}_{t\in\mathbb{R}}$ as shown in $(H_1)$  is a uniformly pullback $\mathscr D$-absorbing family of the process  $U_\lambda(t,\tau)$.     Then the family $\mathcal E_\lambda=\{E_\lambda(t)\}_{t\in\mathbb{R}}$ given by  Theorem \ref{42} is a pullback $\mathscr D$-exponential attractor of the process  $U_\lambda(t,\tau)$.
 \end{corollary}

\begin{proof}
 By  Theorem \ref{42}, it is enough to prove that  $\mathcal{E}_\lambda$ pullback attracts every family $\mathcal{D}\subset \mathscr D$ at an exponential rate.
For any  $\mathcal D\in \mathscr D$, there exists a constant  $e(\mathcal D)>0$ such that
\begin{equation*}
  U(t, t-\tau)D(t-\tau)\subset  B(t), \ \ \forall \tau\geq e(\mathcal D), \forall t\in \mathbb R.
\end{equation*}
Then it follows from \eqref{14.7} that
\begin{equation*}
  \begin{split}
    &\mathrm{dist}_{X_t}\left(U_\lambda(t,t-\tau)D(t-\tau), E_\lambda(t)\right)\\
  \leq\ & \mathrm{dist}_{X_t}\left(U_\lambda(t, t-\tau+e(\mathcal D))U_\lambda(t-\tau+e(\mathcal D), t-\tau)D(t-\tau), E_\lambda(t)\right) \\
       \leq\ & \mathrm{dist}_{X_t}\left(U_\lambda(t, t-\tau+e(\mathcal D))B(t-\tau+e(\mathcal D)), E_\lambda(t)\right)\\
     \leq\ & C(t)e^{\beta e(\mathcal D)}e^{-\beta\tau}, \ \ \forall t\in\mathbb{R}, \ \tau\geq e(\mathcal D)+\tau_t.
  \end{split}
\end{equation*}
This completes the proof.
\end{proof}

\begin{corollary}\label{ea2}   Let   Assumption \ref{assumptionea} be valid,  and the process $U_\lambda(t, \tau)$ also possess  a uniformly pullback $\mathscr D$-absorbing family $\mathcal D_0=\{D_0(t)\}_{t\in\mathbb{R}}\in \mathscr D$ satisfying the following conditions:
\begin{enumerate}[(i)]
  \item there are positive constants $\kappa$, $\tau_1$ and $C_0$ such that
   \begin{equation}\label{14.8}
   \mathrm{dist}_{X_t}\left(U_\lambda(t, t-\tau)D_0(t-\tau), B(t)\right)\leq C_0 e^{-\kappa \tau}, \ \ \forall t\in \mathbb R, \ \tau\geq \tau_1;
   \end{equation}
  \item there is a positive constant $\mathcal R >R_0$ such that $D_0(t)\subset \mathbb B_t(\mathcal R)$ for all $t\in \mathbb R$ and
      \begin{equation}\label{14.9'}
      \|U_\lambda(t, t-\tau)x- U_\lambda(t, t-\tau)y\|_{X_t}\leq C_1 e^{\gamma \tau}\|x-y\|_{X_{t-\tau}},
      \end{equation}
       for all $x, y\in \mathbb B_{t-\tau}(\mathcal R)$, $\tau\geq 0$ and $t\in \mathbb R$, where $C_1$ and $\gamma$ are positive constants depending only on $\mathcal R$.
     \end{enumerate}
  Then   the family $\mathcal E_\lambda=\{E_\lambda(t)\}_{t\in\mathbb{R}}$ given by  Theorem \ref{42} is a pullback $\mathscr D$-exponential attractor of the process  $U_\lambda(t,\tau)$.
 \end{corollary}
\begin{proof}
For any given $t\in \mathbb R$, we take
\begin{equation*}
  \theta=\frac{\kappa}{2(\gamma+\kappa)}\ \ \hbox{and}\ \ \beta'=\min\left\{\frac{\kappa}{2}, \frac{\kappa \beta}{2(\gamma+\kappa)}\right\}.
\end{equation*}
A simple calculation shows that
\begin{equation*}
  \theta\in (0, 1), \ \ \beta'>0, \ \ \gamma\theta+\kappa\theta-\kappa=-\frac{\kappa}{2}.
\end{equation*}
Due to  $\mathcal D_0\in \mathscr D$, by Definition \ref{Absorbing}, there exists a positive constant $e_1>\tau_1$ such that
\begin{equation*}
U_\lambda(t-\tau\theta, t-\tau)D_0(t-\tau)\subset D_0(t-\tau\theta)\subset \mathbb B_{t-\tau\theta}(\mathcal R)
\end{equation*}
for all $\tau\geq (1-\theta)^{-1}e_1$ and $t\in \mathbb R$.  Thus  it follows from formulas \eqref{14.7}-\eqref{14.9'} and the fact (see $(H_1)$)
\begin{equation*}
B(t-\tau\theta)\subset \mathbb B_{t-\tau\theta}(R_0)\subset \mathbb B_{t-\tau\theta}(\mathcal R)
\end{equation*}
that
\begin{equation}\label{14.10}
\begin{split}
& \mathrm{dist}_{X_t}\left(U_\lambda(t, t-\tau)D_0(t-\tau), E_\lambda(t)\right)\\
\leq\ & \mathrm{dist}_{X_t}\left(U_\lambda(t, t-\tau\theta)U_\lambda(t-\tau\theta, t-\tau)D_0(t-\tau),U_\lambda(t, t-\tau\theta)B(t-\tau\theta)\right)\\
& +  \mathrm{dist}_{X_t}\left(U_\lambda(t, t-\tau\theta)B(t-\tau\theta), E_\lambda(t)\right)\\
\leq\ & C_1 e^{\gamma \tau \theta} \mathrm{dist}_{X_{t-\tau\theta}}\left(U_\lambda(t-\tau\theta, t-\tau)D_0(t-\tau), B(t-\tau\theta)\right)+ C(t)e^{-\beta \tau\theta}\\
\leq \ & C_1C_0  e^{\left(\gamma \theta+\kappa\theta-\kappa \right)\tau}+ C(t)e^{-\beta \tau \theta}\\
 \leq \ & C(t)e^{-\beta' \tau}, \ \ \forall  \tau\geq \max\{ \theta^{-1}\tau_t, (1-\theta)^{-1} e_1\}, \ t\in \mathbb R.
\end{split}
\end{equation}
That is, $\mathcal{E}_\lambda$ pullback attracts $\mathcal{D}_0$ at an exponential rate.
Then  repeating the same proof  as   Corollary \ref{ea1} (replacing $B(t)$ there by $D_0(t)$), we obtain that $\mathcal E_\lambda$ is a pullback $\mathscr D$-exponential attractor of the process $U_\lambda(t, \tau)$.
\end{proof}

 \begin{proof}[\textbf{Proof of Theorem \ref{42}}] For clarity, we divide the proof into four steps.

 \textbf{Step 1.} (Construction of the family $\{E_\lambda(n)\}_{n\leq n_*}$, with $\lambda\in \Lambda\backslash \Lambda_0$)  For simplicity and  without loss of generality, we take $T=1$.  It follows from Assumption \ref{assumptionea}  that  for each $\lambda\in \Lambda$, the discrete process $U_\lambda(m, n): X_n\rightarrow X_m$   possesses the following properties:
\begin{align}
&\bigcup_{\lambda\in \Lambda}U_\lambda(m, n)B(n)\subset B(m), \ \ \forall m\geq n\in \mathbb Z,\label{14.11}\\
& \sup_{\lambda\in \Lambda}\|U_\lambda(n, n-1)x- U_\lambda(n, n-1)y\|_{X_n}\leq L_1\|x-y\|_{X_{n-1}},\ \ \forall x, y\in B(n-1), \ n\in \mathbb Z, \label{14.12}
\end{align}
and
\begin{align}
&\|U_\lambda(n, n-1)x- U_\lambda(n, n-1)y\|_{X_n}\leq \eta\|x-y\|_{X_{n-1}}+n_Z\left(K^\lambda_n x-K^\lambda_ny\right),\label{14.13}\\
&  \sup_{\lambda\in \Lambda}\|K^\lambda_n x-K^\lambda_ny\|_Z\leq L\|x-y\|_{X_{n-1}}, \ \   \forall  x, y\in B(n-1), n\leq n_*,  \label{14.14}
\end{align}
 where $n_*\in \mathbb Z$ satisfies $t_0-n_* \in [0, 1)$.

Repeating  the similar arguments as   Theorem 2.5 in \cite{Ly-YJDDE} (or Theorem 2.3 in \cite{Y-LyPEA}), one easily  infers from formulas \eqref{14.11}-\eqref{14.14} that
\begin{equation}\label{14.15}
N_n(k):= N_n\left(U_\lambda(n, n-k)B(n-k), (2\eta)^k R_0\right)\leq \left[m_Z\left(\frac{2L}{\eta}\right)\right]^k, \ \ \forall  n\leq n_*, k\geq 1, \lambda\in \Lambda,
\end{equation}
 where $N_n(B, \e)$ denotes the cardinality of minimal covering of the set $B\subset X_n$ by the closed subsets of $X_n$ with diameter $\leq 2\e$.

It follows from \eqref{14.15} and \eqref{14.11} that there exists  a finite subset $V^\lambda_k(n)$ of $X_n$   for each $\lambda\in \Lambda$, $n\leq n_*$ and   $k\geq 1$, which possesses the following properties:
\begin{align}
& \mathrm{Card}\left(V^\lambda_k(n)\right)\leq \left[m_Z\left(\frac{2L}{\eta}\right)\right]^k,\label{14.18}\\
& V^\lambda_k(n)\subset U_\lambda(n, n-k)B(n-k)\subset B(n)\subset X_n,\label{14.16}\\
& U_\lambda(n, n-k)B(n-k)\subset \bigcup_{h\in V^\lambda_k(n)} \mathbb B_n\left(h; 2R_0(2\eta)^k\right). \label{14.17}
\end{align}
Given $\lambda_0\in \Lambda$,  for  every $\lambda\in \Lambda\backslash \Lambda_0$ and  $n\leq n_*$,   we define by induction the sets
\begin{equation}\label{14.19}
\left\{
  \begin{array}{ll}
    E^\lambda_1(n)=V^\lambda_1(n),  \\
~\\
    E^\lambda_k(n)= V^\lambda_k(n)\cup U_\lambda(n, n-1)E^\lambda_{k-1}(n-1),\ \ k\geq 2,  \\
~\\
    E_\lambda(n)=\left[\bigcup_{k\geq 1}E^\lambda_k(n)\right]_{X_n}.
  \end{array}
\right.
\end{equation}
It follows from \eqref{14.11} and \eqref{14.16}-\eqref{14.17} that
\begin{align}
& E^\lambda_k(n)=\bigcup_{l=0}^{k-1}U_\lambda(n, n-l)V^\lambda_{k-l}(n-l)\subset U_\lambda(n, n-k)B(n-k),\label{14.20}\\
& U_\lambda(n+1, n)E^\lambda_k(n)\subset E^\lambda_{k+1}(n+1),\label{14.21}\\
& E_\lambda(n)\subset U_\lambda(n, n-1)B(n-1)\subset B(n),     \label{14.22}
\end{align}
for all $\lambda\in \Lambda\backslash \Lambda_0$, $n\leq n_*$, $k\geq 1$. Moreover, we infer from \eqref{14.18} and \eqref{14.20} that
\begin{equation}\label{14.25}
  \mathrm{Card}\left(E^\lambda_k(n)\right)\leq \sum_{l=0}^{k-1} \mathrm{Card}\left(V^\lambda_{k-l}(n-l)\right)\leq \left[m_Z\left(\frac{2L}{\eta}\right)\right]^{k+1},  \ \ \forall  \lambda\in \Lambda\backslash \Lambda_0, n\leq n_*, k\geq 1.
\end{equation}
 \medskip

\textbf{Step 2.} (The properties of the family $\{E_\lambda(n)\}_{n\leq n_*}$, with $\lambda\in \Lambda\backslash \Lambda_0$) We show that, for every $\lambda\in \Lambda\backslash \Lambda_0$, the family $\{E_\lambda(n)\}_{n\leq n_*}$ is of the following properties:

(i) (Semi-invariance)  It follows from formulas  \eqref{14.12},   \eqref{14.19} and   \eqref{14.21} that
\begin{equation}\label{14.26}
U_\lambda(n, l)E_\lambda(l)\subset \Big[\bigcup_{k\geq 1}U_\lambda(n, l)E^\lambda_k(l)\Big]_{X_n}\subset\Big[\bigcup_{k\geq 1}E^{\lambda}_{k+n-l}(n)\Big]_{X_n}\subset E_\lambda(n), \ \ \forall l \leq n\leq n_*.
\end{equation}

(ii) (Pullback exponential attractiveness)  Formula \eqref{14.19} shows  $V^\lambda_k(n)\subset E_\lambda(n)$  for all $n\leq n_*$ and $k\geq 1$. Thus  by \eqref{14.17}, we have
\begin{equation}\label{14.27}
\begin{split}
& \mathrm{dist}_{X_n}\left(U_\lambda(n, n-k)B(n-k), E_\lambda(n)\right)\\
\leq \ & \mathrm{dist}_{X_n}\left(U_\lambda(n, n-k)B(n-k), V^\lambda_k(n)\right)\\
\leq\ & 2(2\eta)^k R_0, \ \ \forall k\geq 1, \ n\leq n_*.
\end{split}
\end{equation}

(iii) (The compactness and the boundedness of the fractal dimension)  For any $\e\in(0, 1)$, there exists  unique $k_\e\in \mathbb N$ such that
\begin{equation}\label{14.28}
2(2\eta)^{k_\e} R_0<\e\leq 2(2\eta)^{k_\e-1} R_0.
\end{equation}
Obviously, $\lim_{\e\rightarrow 0^+}k_\e= +\infty$. It follows from  \eqref{14.11} and \eqref{14.20} that
\begin{equation*}
E^\lambda_k(n)\subset U_\lambda(n, n-k)B(n-k)\subset U_\lambda(n, n-k_\e)B(n-k_\e), \ \ \forall k\geq k_\e, \ n\leq n_*,
\end{equation*}
which implies
\begin{equation*}
E_\lambda(n)\subset \Big(\bigcup_{k<k_\e} E^\lambda_k(n)\Big) \bigcup\Big[U_\lambda(n, n-k_\e)B(n-k_\e) \Big]_{X_n}, \ \ \forall n\leq n_*.
\end{equation*}
Thus by \eqref{14.15}, \eqref{14.25} and \eqref{14.28}, we obtain
\begin{equation}\label{14.29}
\begin{split}
N_n\left(E_\lambda(n), \e\right)&\leq N_n\left(E_\lambda(n), 2(2\eta)^{k_\e} R_0\right)\\
 &\leq \sum_{k<k_\e} \mathrm{Card}\left(E^\lambda_k(n)\right)+ N_n\left(U_\lambda(n, n-k_\e)B(n-k_\e), 2(2\eta)^{k_\e} R_0\right)\\
&\leq \sum_{k<k_\e} \left[m_Z\left(\frac{2L}{\eta}\right)\right]^{k+1}+ N_n(k_\e)\leq 2\left[m_Z\left(\frac{2L}{\eta}\right)\right]^{k_\e+1}<+\infty.
\end{split}
\end{equation}
  Formula  \eqref{14.29} shows that $E_\lambda(n)$ is a compact subset of $X_n$ for the arbitrariness of $\e\in(0, 1)$. Moreover, by virtue of estimates \eqref{14.28}-\eqref{14.29} and a simple calculation, we have
  \begin{equation*}
  \frac{\ln N_n\left(E_\lambda(n),\epsilon\right)}{\ln{(1/\e)}}\leq \frac{(k_\e+1)\ln\left[m_Z\left(\frac{2L}{\eta}\right)\right]+\ln 2 }{(k_\e-1)\ln(1/ 2\eta)-\ln(2R_0)}, \ \  \forall \e\in(0, 1),
  \end{equation*}
  which implies
\begin{equation}\label{14.30}
\mathrm{dim}_f\left(E_\lambda(n); X_n \right)=\limsup_{\e\rightarrow 0^+}\frac{\ln N_n\left(E_\lambda(n),\epsilon\right)}{\ln{(1/\e)}}
 \leq  \Big[\ln\Big(\frac{1}{2\eta}\Big)\Big]^{-1}\ln \left[ m_Z\left(\frac{2L}{\eta}\right)\right],\ \ \forall n\leq n_*.
\end{equation}

 \textbf{Step 3.} (The H\"{o}lder continuity of the family  $\{E_\lambda(n)\}_{n\leq n_*}$    at the point $\lambda_0$) Taking into account  $\lambda_0\in \Lambda\setminus \Lambda_0$,  by \eqref{14.20} we have
 $E_k^{\lambda_0}(n) \subset U_{\lambda_0}(n,n-k)B(n-k)$  for all
 $k\geq 1$ and $n\leq n_*$,  so there must be   a subset  $\tilde{E}_k(n-k)\subset B(n-k)$ satisfying
 \begin{equation}\label{14.36}
 U_{\lambda_0}(n,n-k) \tilde{E}_k(n-k)=E^{\lambda_0}_k(n),\ \ \mathrm{Card}\left(\tilde{E}_k(n-k)\right)= \mathrm{Card}\left(E^{\lambda_0}_k(n)\right).
\end{equation}
By \eqref{14.21}  and \eqref{14.36},
\begin{equation}\label{14.37}
  \begin{split}
   &  U_{\lambda_0}(n+1,n-k) \tilde{E}_k(n-k)
      =U_{\lambda_0}(n+1, n) U_{\lambda_0}(n,n-k) \tilde{E}_k(n-k)\\
     =\ & U_{\lambda_0}(n+1, n)E^{\lambda_0}_k(n)
     \subset E^{\lambda_0}_{k+1}(n+1)
     =U_{\lambda_0}(n+1,n-k)\tilde{E}_{k+1}(n-k),
  \end{split}
\end{equation}
which implies
\begin{equation}\label{14.38}
 \tilde{E}_k(n-k)\subset \tilde{E}_{k+1}(n-k),\ \ \forall k\geq 1, n\leq n_*.
\end{equation}

 For each $\lambda\in\Lambda_0, n\leq n_*$ and $k\geq 1$, we define the set
\begin{equation}\label{14.39}
\tilde E^\lambda_k(n):=U_{\lambda}(n,n-k) \tilde{E}_k(n-k).
\end{equation}
Obviously, $\tilde E^\lambda_k(n)\subset U_{\lambda}(n,n-k) B(n-k)$, and it follows from formula \eqref{14.36} and estimate \eqref{14.25} that
\begin{equation}\label{14.40}
   \mathrm{Card}\left(\tilde E^\lambda_k(n)\right)\leq  \mathrm{Card}\left(\tilde{E}_k(n-k)\right)\leq\left[m_Z\left(\frac{2L}{\eta}\right)\right]^{k+1}, \ \ \forall k\geq 1, n\leq n_*.
\end{equation}
Therefore,  by the definition of $\Gamma(\lambda,\lambda_0)$ (see \eqref{14.5}) and Assumption  \ref{assumptionea},
\begin{equation*}
  \begin{split}
      & \|U_\lambda(n,n-k)x
      -U_{\lambda_0}(n,n-k)x\|_{X_n}\\
  \leq & \|U_\lambda(n,n-1)U_\lambda(n-1,n-k)x
  -U_\lambda(n, n-1)U_{\lambda_0}(n-1,n-k)x\|_{X_n}\\
     & +  \|U_\lambda(n, n-1)U_{\lambda_0}(n-1,n-k)x
     -U_{\lambda_0}(n, n-1)U_{\lambda_0}(n-1,n-k)x\|_{X_n}\\
     \leq  & L_1 \|U_\lambda(n-1,n-k)x
     -U_{\lambda_0}(n-1,n-k)x\|_{X_{n-1}}
     +\Gamma(\lambda,\lambda_0)\\
     \leq  &\cdots\leq\Gamma(\lambda,\lambda_0)\left(L_1^{k-1}
     +\cdots+1\right)\\
     \leq & \Gamma(\lambda,\lambda_0)L_1^k,\ \ \forall x\in B(n-k), \ n\leq n_*,  k\geq 1,
   \end{split}
\end{equation*}
that is,
\begin{equation}\label{14.41}
\sup_{x\in B(n-k)}\sup_{n\leq n_*}\|U_\lambda(n,n-k)x
-U_{\lambda_0}(n,n-k)x\|_{X_n}
\leq \Gamma(\lambda,\lambda_0)L_1^k,\ \ \forall k\geq 1.
\end{equation}
When $\lambda\in \Lambda_0$, which means  $0<\Gamma(\lambda,\lambda_0)<1$, we let
\begin{equation*}
 k_\Gamma: = \frac{-\log_{L_1} (\Gamma(\lambda,\lambda_0))}{1-\log_{L_1} (2\eta)}\ \ \hbox{and}\ \    \gamma:=\frac{-\log_{L_1} (2\eta)}{1-\log_{L_1} (2\eta)}.
\end{equation*}
A simple calculation shows that
\begin{equation*}
 \Gamma(\lambda,\lambda_0) L_1^{k_\Gamma}=(2\eta)^{k_\Gamma}
 =\left[\Gamma(\lambda,\lambda_0)\right]^\gamma.
\end{equation*}
Hence, it follows  from formulas \eqref{14.36}, \eqref{14.39},  \eqref{14.41} and  $\tilde E_k(n-k)\subset B(n-k)$ that
\begin{equation}\label{14.42}
\begin{split}
&\sup_{n\leq n_*}\mathrm{dist}^{symm}_{X_n}
\left(\tilde{E}^\lambda_k(n),
E^{\lambda_0}_k(n) \right)\\
      =\ &\sup_{n\leq n_*}\mathrm{dist}^{symm}_{X_n}
      \left(U_{\lambda}(n,n-k) \tilde{E}_k(n-k), U_{\lambda_0}(n,n-k) \tilde{E}_k(n-k) \right)\\
 \leq\  &  \Gamma(\lambda,\lambda_0)L_1^k\leq \left[\Gamma(\lambda,\lambda_0)\right]^\gamma,\ \ 1\leq k\leq  k_\Gamma.
\end{split}
\end{equation}

For each $b\in B(n-k)$, by the finiteness of the subset   $E_k^{\lambda_0}(n)$, there must be  a $b_0\in E_k^{\lambda_0}(n)$  such that
\begin{equation}\label{14.43}
  \begin{split}
 &\|U_{\lambda_0}(n,n-k)b-b_0\|_{X_n}\\
      =\ &\mathrm{dist}_{X_n}
     \left(U_{\lambda_0}(n,n-k)b, E^{\lambda_0}_k(n) \right)\\
    \leq\ &  \mathrm{dist}_{X_n}
    \left(U_{\lambda_0}(n,n-k)B(n-k), V^{\lambda_0}_k(n) \right)\\
    \leq\ & 2R_0(2\eta)^k,\ \ \forall k\geq 1, \ n\leq n_*,
  \end{split}
\end{equation}
where we have used formula \eqref{14.17}. The combination of \eqref{14.41}-\eqref{14.43} yields, for any $b\in B(n-k)$,
\begin{equation*}
  \begin{split}
   &\mathrm{dist}_{X_n}
   \left(U_{\lambda}(n,n-k)b, \tilde{E}^\lambda_k(n)\right)\\
   \leq\ & \|U_{\lambda}(n,n-k)b
   -U_{\lambda_0}(n,n-k)b\|_{X_n}+ \|U_{\lambda_0}(n,n-k)b-b_0\|_{X_n}
      +\mathrm{dist}_{X_n}\left(b_0, \tilde{E}^\lambda_k(n)\right)\\
    \leq\ & \Gamma(\lambda,\lambda_0) L_1^{k}+ 2R_0(2\eta)^k+ \left[\Gamma(\lambda,\lambda_0)\right]^\gamma\\
    \leq\ & (2R_0+2)(2\eta)^k, \ \ 1\leq k\leq k_\Gamma.
  \end{split}
\end{equation*}
By the arbitrariness of $b\in B(n-k)$,
\begin{equation}\label{14.44}
  \mathrm{dist}_{X_n}
  \left(U_{\lambda}(n,n-k)B(n-k), \tilde{E}^\lambda_k(n)\right)\leq  (2R_0+2)(2\eta)^k, \ \ \forall  1\leq k\leq k_\Gamma, \ n\leq n_*,\ \lambda\in \Lambda_0.
\end{equation}

For each $\lambda\in \Lambda_0$ and $n\leq n_*$, we define by induction the sets
\begin{equation*}
   E^\lambda_k(n)
   =\left\{
   \begin{array}{ll}
   \tilde{E}_k^\lambda(n), &\ \ 1\leq k\leq k_\Gamma,\\
    V_k^\lambda(n)\cup U_\lambda(n, n-1) E_{k-1}^\lambda(n-1), &\ \  k>k_\Gamma,
                       \end{array}
                     \right.
\end{equation*}
 and
  \begin{equation}\label{14.45}
 E_\lambda(n)= \Big[\bigcup_{k\geq1}
 E^\lambda_k(n)\Big]_{X_n}.
 \end{equation}
 It follows from formulas \eqref{14.18}-\eqref{14.17} and  \eqref{14.40} that
 \begin{align}
   & E^\lambda_k(n)\subset U_\lambda(n,n-k)B(n-k), \label{14.46}\\
  & U_\lambda(n+1, n)E^\lambda_k(n)\subset E^\lambda_{k+1}(n+1), \label{14.47}\\
  &  E_\lambda(n)\subset U_\lambda(n, n-1)B(n-1)\subset B(n), \label{14.48}\\
  &\mathrm{Card}\left(E^\lambda_k(n)\right) \leq \Big[m_Z\Big(\frac{2L}{\eta}\Big)\Big]^{k+1}
  \label{14.49}
\end{align}
 for all $\lambda\in \Lambda$, $n\leq n_*$ and $k\geq 1$.

 For every $\lambda\in \Lambda_0$, repeating the same arguments as in  Step 2, we obtain that the family $\{E_\lambda(n)\}_{n\leq n_*}$ also possesses the properties (i)-(iii) as shown in Step 2.
  Furthermore, it follows from formulas \eqref{14.17} and \eqref{14.44} that
 \begin{equation}\label{14.50}
\begin{split}
&\mathrm{dist}_{X_n}
\left(U_{\lambda}(n,n-k)B(n-k), E^\lambda_k(n)\right)\\
\leq\ &\begin{cases} \mathrm{dist}_{X_n}
\left(U_\lambda(n,n-k)B(n-k), \tilde{E}^\lambda_k(n)\right),\ \ 1\leq k\leq k_\Gamma,\\
\mathrm{dist}_{X_n}
\left(U_\lambda(n,n-k)B(n-k), V^\lambda_k(n)\right),\ \ k> k_\Gamma,
\end{cases}\\
\leq\ &(2R_0+2)(2\eta)^k, \ \ \forall n\leq n_*, k\geq 1.
\end{split}
\end{equation}

Now,  we show the H\"{o}lder continuity of the family  $\{E_\lambda(n)\}_{n\leq n_*}$    at the point $\lambda_0$.

For  any $\lambda\in \Lambda_0$, $n\leq n_*$ and any  $a\in \cup_{k\geq 1} E_k^\lambda(n)$,   there must be $a\in E_k^\lambda(n)$ for some $k\geq 1$. \\
When $1\leq k\leq k_\Gamma, E_k^\lambda(n)=\tilde{E}^\lambda_k(n)$, by \eqref{14.42}  we obtain
\begin{equation*}
 \mathrm{dist}_{X_n}\left(a, E_{\lambda_0}(n)\right)\leq \mathrm{dist}_{X_n}\left(a, E_k^{\lambda_0}(n)\right) \leq \mathrm{dist}_{X_n}\left(\tilde{E}^\lambda_k(n), E_k^{\lambda_0}(n)\right)\leq \left[\Gamma(\lambda, \lambda_0)\right]^\gamma.
\end{equation*}
When $k>k_\Gamma$, by \eqref{14.11} and  \eqref{14.46},
\begin{equation*}
 a\in E^\lambda_k(n)\subset  U_{\lambda}(n,n-k)B(n-k)=U_{\lambda}(n,n-[k_\Gamma])
 U_{\lambda}(n-[k_\Gamma],n-k)B(n-k),
\end{equation*}
then there exists an element $b\in U_{\lambda}(n-[k_\Gamma],n-k)B(n-k)$ such that  $a=U_{\lambda}(n,n-[k_\Gamma]) b$,
 and by \eqref{14.41} and \eqref{14.27},
\begin{equation*}
  \begin{split}
 &\mathrm{dist}_{X_n}\left(a, E_{\lambda_0}(n)\right)\\
 \leq\ & \mathrm{dist}_{X_n}
 \left(U_{\lambda}(n,n-[k_\Gamma])b,
        E_{\lambda_0}(n)\right) \\
       \leq\ & \|U_{\lambda}(n,n-[k_\Gamma])b- U_{\lambda_0}(n,n-[k_\Gamma])b\|_{X_n}+ \mathrm{dist}_{X_n}
        \left(U_{\lambda_0}(n,n-[k_\Gamma])b, E_{\lambda_0}(n)\right)\\
       \leq\ &\Gamma(\lambda, \lambda_0) L_1^{[k_\Gamma]}+(2R_0+2)(2\eta)^{[k_\Gamma]}
       \leq C\left[\Gamma(\lambda, \lambda_0)\right]^\gamma,
  \end{split}
\end{equation*}
where $C=1+\frac{R_0+1}{\eta}$ and $[k_\Gamma]$ denotes the integer part of $k_\Gamma$. By the arbitrariness of $a\in \cup_{k\geq 1} E_k^\lambda(n)$, we obtain
\begin{equation}\label{14.51}
\mathrm{dist}_{X_n}\left( E_\lambda(n),E_{\lambda_0}(n)\right)
 =\mathrm{dist}_{X_n}
 \left(\cup_{k\geq1}E^\lambda_k(n), E_{\lambda_0}(n)\right)\leq C\left[\Gamma(\lambda, \lambda_0)\right]^\gamma.
\end{equation}
Repeating the proof of \eqref{14.51} (changing the position of $\lambda$ and $\lambda_0$) and making use of \eqref{14.42} and \eqref{14.50}, we have
\begin{equation}\label{14.520}
 \mathrm{dist}_{X_n}\left( E_{\lambda_0}(n), E_\lambda(n)\right)\leq C\left[\Gamma(\lambda, \lambda_0)\right]^\gamma.
\end{equation}
The combination of \eqref{14.51} and \eqref{14.520} gives
\begin{equation}\label{14.52}
 \sup_{n\leq n_*}\mathrm{dist}^{symm}_{X_n}\left( E_\lambda(n),E_{\lambda_0}(n)\right)\leq C\left[\Gamma(\lambda, \lambda_0)\right]^\gamma,\ \  \forall  \lambda\in \Lambda_0.
\end{equation}

\textbf{Step 4.} (The structure  of the  family $\mathcal E_\lambda=\{E_\lambda(t)\}_{t\in\mathbb R}$, with $\lambda\in \Lambda$, and its properties) For any $t\in \mathbb{R}$, when $t\leq t_0$, there exists  unique integer  $n_t=[t]\leq n_*$ such that $t=n_t+s_t$ with $s_t\in [0, 1)$;   when $t>t_0$, we take $n_t\equiv n_*$ and $t=n_*+s_t$.   Let
\begin{equation}\label{14.53}
  E_\lambda(t)=U_\lambda(t, n_t)E_{\lambda}(n_t), \ \ \forall t\in\mathbb{R}, \lambda\in \Lambda.
\end{equation}
By \eqref{14.22}, \eqref{14.48} and condition \eqref{14.1} we have  $E_\lambda(t)\subset  B(t)$ for all $t\in\mathbb{R}$ and $\lambda\in \Lambda$.   We show that  for every $\lambda\in \Lambda, \mathcal E_\lambda=\{E_\lambda(t)\}_{t\in\mathbb R}$ is the desired family of Theorem \ref{42}.

(i) (Semi-invariance) For every $t\leq r\in \mathbb R$, we have   $n_t\leq n_r$. When $t\leq t_0$, we have $n_t\leq n_*$, and it follows from the semi-invariance \eqref{14.26} of $\{E_\lambda(n)\}_{n\leq n_*}$ and formula \eqref{14.53} that
\begin{equation*}
\begin{split}
  U_\lambda(r,t)E_\lambda(t)
  &=U_\lambda(r,t)U_\lambda(t,n_t)
  E_\lambda(n_t)\\
  &=U_\lambda(r,n_r)U_\lambda(n_r,n_t)
  E_\lambda(n_t)\\
&\subset U_\lambda(r,n_r)E_\lambda(n_r)
=E_\lambda(r), \ \ \forall \lambda\in \Lambda;
\end{split}
\end{equation*}
When $t>t_0$, $n_t\equiv n_*$ and we also have
 \begin{equation*}
  U_\lambda(r,t)E_\lambda(t)=U_\lambda(r,t)U_\lambda(t,n_*)
  E_\lambda(n_*)=U_\lambda(r,n_*)E_\lambda(n_*)
   =E_\lambda(r), \ \ \forall \lambda\in \Lambda.
\end{equation*}

(ii) (The compactness and the boundedness of the  fractal dimension)  For each $t\in \mathbb R$, let
\begin{equation*}
L_t= \left\{
       \begin{array}{ll}
         L_1, & t\leq t_0, \\
         L_1^{1+t-t_0}, & t>t_0.
       \end{array}
     \right.
\end{equation*}
We infer from condition $(H_3)$ that
\begin{equation}\label{14.54}
  \sup_{\lambda\in \Lambda}\|U_\lambda(t,n_t)x
  -U_\lambda(t,n_t)y\|_{X_t}\leq L_t\|x-y\|_{X_{n_t}},\ \ \forall x, y\in B(n_t),
\end{equation}
that is, the mapping $U_\lambda(t, n_t): B(n_t)\rightarrow X_t$ is Lipschitz continuous. Therefore, the set  $E_\lambda(t)$ is compact in $X_t$ for each $\lambda\in \Lambda$, and by \eqref{14.30},
\begin{equation*}
   \mathrm{dim}_f\left(E_\lambda(t); X_t\right)\leq \mathrm{dim}_f\left(E_\lambda(n_t); X_{n_t} \right)\leq \left[\ln\left(\frac{1}{2\eta}\right)\right]^{-1}\ln \left[ m_Z\left(\frac{2L}{\eta}\right)\right].
\end{equation*}

(iii) (Pullback exponential attractiveness)  For each $t\in \mathbb R$, let
\begin{equation*}
\tau_t= \left\{
       \begin{array}{ll}
         3, & t\leq t_0, \\
         3+t-t_0, & t>t_0.
       \end{array}
     \right.
\end{equation*}
For every $\tau\geq \tau_t$, there exists  unique integer $k_\tau\in \mathbb N$ such that $\tau\in [k_\tau+ 2, k_\tau+3)$, which implies
\begin{equation}\label{14.55}
n_t-k_\tau-(t-\tau)>1 \ \ \hbox{and}\ \ - k_\tau \leq -\tau + t-n_t+2\leq -\tau+ \tau_t.
\end{equation}
Then  it follows from   \eqref{14.1} and   \eqref{14.55} that
\begin{equation*}
\begin{split}
  U_\lambda(t, t-\tau)B(t-\tau)
=\ & U_\lambda(t, n_t )U_\lambda\left(n_t, n_t-k_\tau\right) U_\lambda\left(n_t-k_\tau, t-\tau\right) B(t-\tau)\\
\subset\ &U_\lambda(t, n_t )U_\lambda\left(n_t, n_t-k_\tau\right)B(n_t-k_\tau),
\end{split}
\end{equation*}
which combining with formulas \eqref{14.54},   \eqref{14.27} and   \eqref{14.53} yields
\begin{equation*}
  \begin{split}
    &\mathrm{dist}_{X_t}
    \left(U_\lambda(t,t-\tau)B(t-\tau), E_{\lambda}(t)\right)\\
\leq\ & \mathrm{dist}_{X_t}
\left(U_\lambda(t,n_t)U_\lambda\left(n_t, n_t-k_\tau\right)B(n_t-k_\tau), U_\lambda(t,n_t)E_\lambda(n_t)\right)\\
  \leq\    & L_t \mathrm{dist}_{X_{n_t}}
  \left(U_\lambda\left(n_t, n_t-k_\tau\right)B(n_t-k_\tau), E_\lambda(n_t)\right) \\
\leq\ & 2L_tR_0 (2\eta)^{k_\tau} =2L_tR_0e^{-\beta k_\tau }
\leq  2L_tR_0e^{\beta \tau_t}e^{-\beta\tau}= C(t)e^{-\beta\tau},
  \end{split}
\end{equation*}
with $\beta=\ln \frac{1}{2\eta}$ and $C(t)=2L_t R_0 e^{\beta \tau_t}$.

(iv) (The continuity of the   family $\mathcal E_\lambda=\{E_\lambda(t)\}_{t\in\mathbb R}$ at the point $\lambda_0$) For each $\lambda\in \Lambda_0$, it follows from  \eqref{14.52}-\eqref{14.54}, \eqref{14.48} and the definition of $\Gamma_t(\lambda, \lambda_0)$ that
\begin{equation*}
  \begin{split}
   & \mathrm{dist}^{symm}_{X_t}\left(E_{\lambda}(t), E_{\lambda_0}(t)\right)\\
     \leq\ & \mathrm{dist}^{symm}_{X_t}
     \left(U_\lambda(t,n_t)E_\lambda (n_t),U_{\lambda_0}(t,n_t)E_\lambda(n_t)\right) \\
    &+  \mathrm{dist}^{symm}_{X_t}
       \left(U_{\lambda_0}(t,n_t)E_\lambda (n_t),  U_{\lambda_0}(t, n_t)E_{\lambda_0 }(n_t)\right)\\
  \leq\  & {\left\{
             \begin{array}{ll}
              \Gamma(\lambda, \lambda_0)
              +L_t\mathrm{dist}^{symm}_{X_{n_t}}
              \left(E_{\lambda}(n_t), E_{\lambda_0}(n_t)\right) , & t\leq t_0, \\
              \sup_{x\in B(n_*)}\|U_\lambda(t, n_* )x- U_{\lambda_0}(t, n_* )x\|_{X_t}
  +L_t\mathrm{dist}^{symm}_{X_{n_t}}
  \left(E_{\lambda}(n_t), E_{\lambda_0}(n_t)\right) , & t>t_0
             \end{array}
           \right.}\\
    \leq\  & C_1(t)\Gamma_t(\lambda, \lambda_0),
  \end{split}
\end{equation*}
with the constant  $C_1(t)=1+L_t$  depending only  on   $t$. This completes the proof.
\end{proof}

\section{Application to the semilinear damped wave equation}
In this section,  we give an application of the above mentioned criteria to the  following semilinear damped  wave equations with perturbed time-dependent speed of propagation
\begin{align}
  & \r_\e(t) u_{tt}+\alpha u_t-\Delta u+f(u)=g(x), \ \ x\in \Omega, t>\tau, \label{5.1}\\
   & u |_{\partial \Omega}=0, \label{5.2}\\
 & u(x, \tau)=u_0(x), \ \ u_t(x, \tau)=u_1(x),\label{5.3}
\end{align}
where $\Omega\subset \mathbb{R}^3$ is a bounded domain with the  smooth boundary $\partial \Omega$, $\alpha>0$ is the damping coefficient, $g\in L^2(\Omega)$, the function $\r_\e(t)=\e \r(t)$ with perturbation parameter  $\e\in (0, 1]$, and the assumptions on $\r(t)$ and the nonlinearity $f(u)$ are the same as in \cite{pata}. For clarity, we quote them as follows.

 \begin{assumption}\label{assumption51} (i)  $\r\in C^1(\mathbb{R})$ is a decreasing bounded function satisfying
\begin{equation}\label{5.4}
\lim_{t\rightarrow +\infty}\r(t)=0 \ \ \hbox{and}\ \ \sup_{t\in\mathbb{R}}\Big[|\r(t)|+|\r'(t)|\Big]\leq L,
\end{equation}
for some constant $L>0$.

(ii) $f\in C^2(\mathbb{R})$ with  $f(0)=0$ satisfies
\begin{equation}\label{5.5}
  |f''(s)|\leq c(1+|s|), \ \ \forall s\in\mathbb{R},
\end{equation}
for some $c\geq 0$, and the dissipation conditions
\begin{equation*}
  \liminf_{|s|\rightarrow \infty}\frac{f(s)}{s}>- \lambda_1,\ \  2f(s)s\geq F(s)-\nu s^2-c_1, \ \ \forall s\in\mathbb{R},
  \end{equation*}
where $F(s)=\int_0^s f(r)\mathrm{d}r$,  $0<\nu<\lambda_1$, $c_1\geq 0$ and $\lambda_1>0$ is the first eigenvalue of the operator $A=-\tr$ with the Dirichlet boundary condition.
\end{assumption}

\subsection{Notations and main results}
Let $V_0=L^2(\Omega)$, with the inner product $(\cdot, \cdot)$ and norm $\|\cdot\|$. For any $\sigma>0$, let the  spaces
\begin{equation*}
  V_\sigma=D(A^{\sigma/2})\ \ \hbox{and}\ \ V_{-\sigma}= \hbox{the dual space of } V_\sigma
\end{equation*}
equipped with the inner product   and the norm
\begin{equation*}
  \ \ \langle w, v\rangle_\sigma=(A^{\sigma/2}w, A^{\sigma/2}v), \ \ \|w\|_\sigma=\|A^{\sigma/2}w\|, \ \ \forall \sigma\in\mathbb{R}.
\end{equation*}
Then $ V_\sigma$  are the Hilbert spaces and $ V_{\sigma_1}\hookrightarrow\hookrightarrow V_{\sigma_2}$ if $\sigma_1>\sigma_2$. In particular,
\[V_2=H^2(\Omega)\cap H_0^1(\Omega),\ \  V_1=H_0^1(\Omega),\ \  \|w\|_1=\|A^{1/2}w\|=\|\nabla w\|.\]
More preciously, one can infer from \cite{Amann} that
\begin{align*}
& V_\sigma\hookrightarrow H^\sigma(\Omega), \ \   \sigma\geq 0,\\
 &V_\sigma\hookrightarrow L^r(\Omega), \ \ 1\leq r\leq  \frac{6}{3-2\sigma}, \ \ L^q(\Omega)\hookrightarrow V_{-\sigma},\ \ q\geq  \frac{6}{3+2\sigma}, \ \  0\leq \sigma<\frac 32.
 \end{align*}
 For each $t\in \mathbb R$, $\sigma\in \mathbb R$ and $\e\in (0, 1]$, let the time-dependent phase spaces
\begin{equation*}
  \mathcal{H}_{t, \sigma}^\e= V_{\sigma+1}\times V_\sigma\ \ \hbox{and}\ \ \mathcal{H}_t^\e=\mathcal{H}_{t, 0}^\e= V_1\times V_0
\end{equation*}
  equipped with the time-dependent product norms
\begin{equation*}
  \|(u, v)\|^2_{\mathcal{H}_{t, \sigma}^\e}=\|u\|^2_{\sigma+1}+\r_\e(t)\|v\|^2_\sigma\ \ \hbox{and}\ \ \|(u, v)\|^2_{\mathcal{H}_t^\e}=\|u\|^2_1+\r_\e(t)\|v\|^2,
\end{equation*}
respectively. In particular, when $\e=1$, we omit the superscript $\e$ and let
\begin{equation*}
 \mathcal{H}_{t, \sigma}= \mathcal{H}_{t, \sigma}^1\ \ \hbox{and}\ \  \mathcal{H}_t=\mathcal{H}_t^1,
\end{equation*}
respectively. A simple calculation shows that
\begin{equation}\label{5.6}
  \e\|(u, v)\|^2_{\mathcal{H}_{t, \sigma}}\leq \|(u, v)\|^2_{\mathcal{H}_{t, \sigma}^\e} \leq \|(u, v)\|^2_{\mathcal{H}_{t, \sigma}}, \ \  \forall \e\in (0, 1],
\end{equation}
and
\begin{equation}\label{a1}
 a\|(u, v)\|^2_{\mathcal{H}_{t, \sigma}}\leq \|(u, v)\|^2_{\mathcal{H}_{t, \sigma}^\e} \leq \|(u, v)\|^2_{\mathcal{H}_{t, \sigma}},\ \ \forall \e\in [a, 1]\subset (0, 1].
\end{equation}
Obviously, all the spaces $\mathcal{H}^\e_t$ are   same as the linear space, while  formula \eqref{5.6} shows  $\|\cdot\|_{\mathcal{H}^\e_t }\sim\|\cdot\|_{\mathcal{H}_t}$ for each $\e\in(0,1]$.

We rewrite problem \eqref{5.1}-\eqref{5.3} at an abstract level
\begin{align}
   & \r_\e(t) u_{tt}+ \alpha u_t +Au +f(u)=g, \ \ t>\tau, \ \ \label{5.7}\\
   & u(\tau)=u_0, \ \ u_t(\tau)=u_1.  \label{5.8}
\end{align}

Under Assumption \ref{assumption51}, repeating    the similar argument as  Theorem 9.1 in \cite{pata},  one  easily  obtains the existence of the continuous  process $U_\e (t, \tau): \mathcal{H}^\e_\tau\rightarrow \mathcal H ^\e_t$ corresponding to problem \eqref{5.7}-\eqref{5.8}, that is
\begin{equation*}
U_\e(t, \tau)z_\tau=z^\e(t), \ \ \forall z_\tau =(u_0,u_1)\in \mathcal{H}^\e_\tau,\ t\geq \tau,
\end{equation*}
where $z_\e(t)=(u_\e(t), u_{\e t}(t))\in C(\mathbb R^+; V_1\times V_0)$ is the weak solution of problem \eqref{5.7}-\eqref{5.8} corresponding to the initial data  $z_\tau \in \mathcal{H}^\e_\tau$, with $\e\in(0, 1]$.
Besides, for any $z_{i\tau}\in \h_{\tau}$ with $\|z_{i\tau}\|_{\h_{\tau}}\leq R, i=1, 2$, and  any interval $[a, 1]\subset (0, 1]$,
\begin{equation}\label{4.10}
\sup_{\e\in [a, 1]}\|U_\e(t, \tau)z_{1\tau}- U_\e(t, \tau)z_{2\tau}\|_{\h_t}\leq \mathcal{Q}_a(R)e^{\mathcal{Q}_a(R) (t-\tau)}\|z_{1 \tau}-z_{2\tau}\|_{\h_\tau},\ \ t\geq \tau,
\end{equation}
hereafter  $\mathcal{Q}_a(\cdot)=a^{-1} \mathcal{Q}(\cdot)$  and $\mathcal Q(\cdot)$ is an increasing positive function independent of $t$.

Let the universe of problem  \eqref{5.7}-\eqref{5.8} be
\begin{equation*}
\mathscr D:=\{\mathcal D=\{D(t)\}_{t\in \mathbb R}\ |\   \emptyset\neq D(t)\subset \mathcal{H}_t, \ \forall t\in \mathbb R\ \ \hbox{and}\ \ \sup_{t\in \mathbb R} \|D(t)\|_{\mathcal{H}_t}<+\infty\}.
\end{equation*}

\begin{theorem}\label{52} Let Assumption \ref{52} be valid. Then the process $U_\e(t, \tau): \mathcal{H}_\tau\rightarrow \mathcal{H}_t$ has a minimal pullback $\mathscr D$-attractor $\mathcal A_\e=\{A_\e(t)\}_{t\in \mathbb R}$  for each $\e\in(0,1]$, and $\mathcal A_\e$ possesses the following properties:
\begin{enumerate}[(i)]
\item the sections $A_\e(t)$ are uniformly bounded in $\mathcal {H}_{t, 1}$, and their fractal dimensions in $\mathcal H_t$ are uniformly bounded, that is,
    \begin{equation*}
       \sup_{t\in \mathbb R}\|A_\e(t)\|_{\mathcal {H}_{t, 1}}<+\infty\ \ \hbox{and}\ \
   \sup_{t\in \mathbb R}\mathrm{dim}_f(A_\e (t), \h_t)<+\infty, \ \ \forall \e\in(0,1];
    \end{equation*}
  \item (Upper semicontinuity)  for any  point $\e_0\in (0,1]$,
\begin{equation*}
  \lim_{\e\rightarrow \e_0}\mathrm{dist}_{\mathcal{H}_t}\left(A_\e(t), A_{\e_0}(t)\right)=0, \ \ \forall t\in\mathbb{R};
\end{equation*}
  \item (Residual continuity)  there exists  a residual subset $\Lambda^*$ of $(0,1]$ such that
      \begin{equation*}
  \lim_{\e\rightarrow \e_0}\mathrm{dist}^{symm}_{\mathcal{H}_t}\left(A_\e(t), A_{\e_0}(t)\right)=0, \ \ \forall \e_0\in \Lambda^*,\  t\in\mathbb{R}.
\end{equation*}
\end{enumerate}
\end{theorem}

\begin{theorem}\label{54} Let Assumption \ref{52} be valid. Then the process $U_\e(t, \tau): \mathcal{H}_\tau\rightarrow \mathcal{H}_t$ has a pullback $\mathscr D$-exponential attractor $\mathcal E_\e=\{E_\e(t)\}_{t\in \mathbb R}$ for each $\e\in(0,1]$,  and $\mathcal E_\e$ possesses the following properties:
\begin{enumerate}[(i)]
\item (Regularity) the sections $E_\e(t)$ are uniformly bounded in $\mathcal {H}_{t, 1}$, that is,
    \begin{equation*}
    \sup_{t\in \mathbb R}\|E_\e(t)\|_{\mathcal {H}_{t, 1}}<\infty, \ \ \forall \e\in(0,1];
    \end{equation*}
  \item (H\"{o}ilder continuity) for any $\e_0\in(0,1]$, there exists  a    $\delta=\delta(\e_0)\in (0, 1)$ such that
\begin{equation*}
\mathrm{dist}^{symm}_{\mathcal{H}_t}\left(E_\e(t), E_{\e_0}(t)\right)\leq C(t)|\e-\e_0|^\gamma,\ \ \forall t\in\mathbb{R},
\end{equation*}
 whenever $|\e-\e_0|<\delta$, where  $C(t)$ and $\gamma: 0<\gamma<1$ are some positive constants.
\end{enumerate}
\end{theorem}

\subsection{Some essential estimates }
In order to  prove Theorem \ref{52} and Theorem \ref{54}, we need  the following lemmas which will play key role later. In this subsection, we always assume the interval $[a, 1]\subset (0, 1]$.

\begin{lemma}\label{lemma1}  Let  Assumption \ref{assumption51} be valid, and  $z_\tau\in \h_\tau$ with $\|z_\tau\|_{\h_\tau}\leq R$. Then  there exist positive constants $\kappa=\kappa(a)$ and  $R_a$ independent of $R$ and $\e$ such that
\begin{equation*}
\sup_{\e\in [a, 1]}\|U_\e(t, \tau)z_\tau\|_{\h_t}\leq \mathcal{Q}_a(R) e^{-\kappa (t-\tau)}+ R_a,\ \ \forall t\geq \tau,
\end{equation*}
and
\begin{equation}\label{a2}
\sup_{\e\in [a, 1]}\int_\tau^\infty \|u_{\e t}(t)\|^2\d t \leq \mathcal{Q}_a(R), \ \ \forall t\geq \tau,
\end{equation}
 hereafter $\left(u_\e(t), u_{\e t}(t) \right)=U_\e(t, \tau)z_\tau$,  $\mathcal{Q}_a(\cdot)$ and $\mathcal{Q}(\cdot)$  are as shown in \eqref{4.10}, and $R_a$ satisfies $\lim_{a\rightarrow 0^+}R_a=+\infty$.
 \end{lemma}
\begin{proof}
Repeating the similar  argument as Lemma 10.3 in \cite{pata}, we obtain
\begin{equation}\label{b1}
a \|U_\e(t, \tau)z_\tau\|_{\h_t}\leq \|U_\e(t, \tau)z_\tau\|_{\h^\e_t}\leq\q(\|z_\tau\|_{\h^\e_\tau}) e^{-\kappa (t-\tau)}+ R_0,\ \ \forall \e\in [a, 1], \  t\geq \tau,
\end{equation}
 and formula \eqref{a2}, that is,  the conclusions of Lemma \ref{lemma1} hold.
\end{proof}

\begin{remark}\label{remark54}
Lemma  \ref{lemma1} shows that the family  $\{\mathbb B_t(R_1)\}_{t\in\mathbb R}\in \mathscr D$,  with $R_1>R_a$, is a uniformly (w.r.t. $\e\in [a, 1]$) pullback $\mathscr D$-absorbing family of the processes $U_\e(t, \tau):\h_\tau\rightarrow \h_t$,  $\e\in [a, 1]$.
\end{remark}

\begin{lemma}\label{lemma55} \cite{Arrieta,Grasselli}
Let Assumption \ref{assumption51} be valid.  Then  $f$ can be split into the sum $f=f_0+f_1$, where $f_0, f_1\in C^2(\mathbb R)$ satisfying
\begin{equation}\label{b3}
f_0(0)=f'_0(0)=0, \ \   f_0(s)s\geq 0,\ \ |f''_0(s)|\leq C(1+|s|),  \ \ \forall s\in \mathbb R,
\end{equation}
and
\begin{equation}\label{b2}
|f'_1 (s)|\leq C, \ \ \forall s\in \mathbb R.
\end{equation}
\end{lemma}

By virtue of Lemma \ref{lemma55},  we split the solution $\left(u_\e(t), u_{\e t}(t) \right)=U_\e(t, \tau)z_\tau$  into the sum
\begin{equation*}
U_\e(t, \tau)z_\tau=U_{\e, 0}(t, \tau)z_\tau+ U_{\e, 1}(t, \tau)z_\tau,
\end{equation*}
where  $U_{\e, 0}(t, \tau)z_\tau=(v_\e(t), v_{\e t}(t))$   solves
\begin{equation*}
{\left\{
  \begin{array}{ll}
    \rho_\e (t) v_{\e tt} +\alpha v_{\e t}+ A v_\e +f_0(v_\e)=0, \ \ t>\tau, \\
    U_{\e, 0}(\tau, \tau)z_\tau=z_\tau,
  \end{array}
\right.}
\end{equation*}
and $U_{\e, 1}(t, \tau)z_\tau=(w_\e(t), w_{\e t}(t))$ solves
\begin{equation*}
{\left\{
  \begin{array}{ll}
   \rho_\e(t) w_{\e tt} +\alpha w_{\e t}+ A w_\e +f(u_\e)-f_0(v_\e)=g, \ \ t>\tau,  \\
    U_{\e, 1}(\tau, \tau)z_\tau=0.
  \end{array}
\right.}
\end{equation*}

\begin{lemma}\label{lemma2}Let  Assumptions \ref{assumption51} be valid, and $z_\tau\in \mathcal{H}_\tau$ with $\|z_\tau\|_{\mathcal{H}_\tau}\leq R$. Then
\begin{align}
& \sup_{\e\in [a, 1]}\|(v_\e(t), v_{\e t}(t))\|_{\h_t}\leq \mathcal{Q}_a(R) e^{-\kappa (t-\tau)}, \label{b6}\\
& \sup_{\e\in [a, 1]}\|(w_\e(t), w_{\e t}(t))\|_{\h_{t, 1/3}}\leq \mathcal{Q}_a(R), \ \ \forall t\geq \tau, \label{b7}
\end{align}
where $\kappa>0$ is as shown in Lemma \ref{lemma1}.
\end{lemma}
\begin{proof} Repeating   the same  proof as   Lemma 11.2 and Lemma 11.3  in \cite{pata}, we obtain
\begin{align*}
& a\|U_{\e, 0}(t, \tau)z_\tau\|_{\h_t}\leq \|U_{\e, 0}(t, \tau)z_\tau\|_{\h^\e_t}\leq  \q(R)e^{-\kappa (t-\tau)},\\
& a\|U_{\e, 1}(t, \tau)z_\tau\|_{\h_{t, 1/3}}\leq \|U_{\e, 1}(t, \tau)z_\tau\|_{\h^\e_{t, 1/3}}\leq \q(R), \ \ \forall \e\in [a, 1], \  t\geq \tau,
\end{align*}
which imply \eqref{b6} and \eqref{b7}.
\end{proof}

On the basis of Lemma \ref{lemma2}, we further give a delicate estimate.

\begin{lemma}\label{lemma3} Let  Assumption \ref{assumption51} be valid, and  $z_\tau\in \h_{\tau, 1/3}$ with $\|z_\tau\|_{\h_\tau}\leq R$. Then
\begin{equation*}
\sup_{\e\in [a, 1]}\|U_\e(t, \tau)z_\tau\|_{\h_{t, 1/3}}\leq \mathcal{Q}_a(\|z_\tau\|_{\h_{\tau, 1/3}}+R)e^{-\kappa (t-\tau)}+\mathcal{Q}_a(R), \ \ \forall   t\geq \tau.
\end{equation*}
\end{lemma}
\begin{proof} For simplicity, we omit the subscript $\e$ and let $u=u_\e$.
For any  $\e\in [a, 1]$ and $\delta\in (0, 1)$, we define the functional along the solution $ (u(t), u_t(t))=U_\e(t, \tau)z_\tau$ as follows
\begin{equation}\label{c3}
\Phi_\e(t)=\rho_\e(t)\|u_t\|^2_{1/3}+\|u\|^2_{4/3}+ 2\langle f(u)-g, A^{\frac13} u\rangle +\delta \left[ 2\rho_\e(t) \langle u_t, A^{\frac13} u\rangle+\alpha \|u\|^2_{1/3}\right],\ \ t\geq \tau.
\end{equation}
It follows from Lemma \ref{lemma1} and condition \eqref{5.5} that
\begin{equation*}
\|f(u)\|\leq C\| 1+|u|^3\|\leq C\left(1+ \|u\|^3_{L^6(\Omega)}\right)\leq C\left(1+\|u\|^3_1\right)\leq \mathcal{Q}_a(R),\ \ t\geq \tau,
\end{equation*}
which implies that
\begin{equation}\label{c4}
2|\langle f(u)-g, A^{\frac13} u\rangle|\leq 2\left[\|f(u)\|+\|g\| \right]\|u\|_{2/3}\leq \frac14 \|u\|^2_{4/3}+ \mathcal{Q}_a(R),\ \ t\geq \tau.
\end{equation}
By condition \eqref{5.4},
\begin{equation}\label{c5}
2\rho_\e(t) | \langle u_t, A^{\frac13} u\rangle|\leq 2\rho_\e(t) \|u_t\|_{1/3}\|u\|_{1/3}
\leq \alpha \|u\|^2_{1/3}+  \frac{L}{\alpha} \rho_\e (t)\|u_t\|^2_{1/3}.
\end{equation}
The combination of \eqref{a1} and \eqref{c3}-\eqref{c5} yields
\begin{equation}\label{c6}
\frac{ a}{2} \|(u(t), u_t(t))\|^2_{\h_{t, 1/3}}-\mathcal{Q}_a(R)\leq \Phi_\e(t)\leq 2 \|(u(t), u_t(t))\|^2_{\h_{t, 1/3}}+\mathcal{Q}_a(R)
\end{equation}
for $\delta>0$ suitably small.
\medskip

Taking the multiplier  $2A^{\frac13} u_t + 2\delta A^{\frac13} u$ in  Eq. \eqref{5.7} yields
\begin{equation}\label{c7}
\begin{split}
&\frac{\d}{\d t}  \Phi_\e (t) + \delta \Phi_\e (t)+\left(2 \alpha -\rho'_\e (t)- 3\delta \rho_\e(t)\right)\|u_t\|^2_{1/3}+\delta \|u\|^2_{4/3}-\delta^2 \alpha \|u\|^2_{1/3}\\
=\ & \left(2\delta^2 \rho_\e(t) + 2\delta \rho'_\e(t)\right) \langle u_t, A^{\frac13}u \rangle + I_1+I_2+I_3,
\end{split}
\end{equation}
where
\begin{align*}
 I_1=2\langle \left[f'_0(u)-f'_0(v)\right]u_t, A^{\frac13} u\rangle,\ \
 I_2= 2\langle f'_0(v)u_t, A^{\frac13} u\rangle,\ \
I_3= 2\langle f'_1(u)u_t, A^{\frac13} u\rangle.
\end{align*}
A simple calculation shows that
\begin{align*} &
\left(2 \alpha -\rho'_\e (t)- 3\delta \rho_\e(t)\right)\|u_t\|^2_{1/3}+\delta \|u\|^2_{4/3}-\delta^2 \alpha \|u\|^2_{1/3}\geq \frac32\alpha\|u_t\|^2_{1/3}+\frac\delta 2\|u\|^2_{4/3},\\
&\left(2\delta^2 \rho_\e(t) + 2\delta \rho'_\e(t)\right) \langle u_t, A^{\frac13}u \rangle \leq 4L\|u_t\|\|A^{\frac13}u \|\leq \frac\delta 4\|u\|^2_{4/3}+\mathcal{Q}_a(R)
\end{align*}
for $\delta>0$ suitably small.
Taking into account the Sobolev embedding:
\begin{equation}\label{c8}
V_1\hookrightarrow L^6(\Omega), \ \ V_{4/3}\hookrightarrow L^{18}(\Omega), \ \ V_{2/3}\hookrightarrow L^{\frac{18}{5}}(\Omega), \ \ V_{1/3}\hookrightarrow L^{\frac{18}{7}}(\Omega),
\end{equation}
and making use of Lemma \ref{lemma1}, formulas \eqref{b2}-\eqref{b7}, we have
\begin{equation*}
\begin{split}
I_1 & \leq C\left(1+ \|u\|_{L^6(\Omega)}+ \|v\|_{L^6(\Omega)} \right)\|w\|_{L^{18}(\Omega)}
     \|u_t\| \| A^{\frac13} u\|_{L^{\frac{18}{5}}(\Omega)}\\
    &\leq C\left(1+ \|u\|_1+ \|v\|_1 \right)\|w\|_{4/3}
     \|u_t\| \|u\|_{4/3}\\
    &\leq \frac \delta 8 \|u\|^2_{4/3}+\mathcal{Q}_a(R),\\
I_2 & \leq C\|v\|_{L^6(\Omega)}\|u_t\|_{L^{\frac{18}{7}}(\Omega)}
       \| A^{\frac13} u\|_{L^{\frac{18}{5}}(\Omega)}\\
    & \leq \alpha \|u_t\|^2_{1/3}+ \mathcal{Q}_a(R) \|v\|^2_1 \|u\|^2_{4/3},
\end{split}
\end{equation*}
and
\begin{equation*}
I_3\leq C\|u_t\|\| A^{\frac13} u\|\leq \frac\delta 8 \|u\|^2_{4/3}+\mathcal{Q}_a(R).
\end{equation*}
Inserting above estimates into \eqref{c7} and making use of estimates \eqref{c5}-\eqref{c6} receive
\begin{equation}\label{c9}
\frac{\d}{\d t} \Phi_\e (t) + \delta \Phi_\e (t)\leq q(t) \Phi_\e (t) +\mathcal{Q}_a(R),\ \ t>\tau,
\end{equation}
where $q(t)= \mathcal{Q}_a(R) \|v(t)\|^2_1$ satisfies (see estimate  \eqref{b6})
\begin{equation*}
\int_\tau^\infty q(s)\d s\leq \mathcal{Q}_a(R)\int_\tau^\infty  \|v(s)\|^2_1\d s \leq \mathcal{Q}_a(R).
\end{equation*}
Applying the Gronwall-type lemma (cf. \cite{patacpaa}) to \eqref{c9} and making use of estimate \eqref{c6} turn out  the conclusions of Lemma \ref{lemma3}.
\end{proof}
\medskip

 For any fixed $\e\in (0, 1]$, $\tau\in \mathbb R$ and $z_\tau\in \h_{\tau, 1/3}$, to avoid using too many symbols we still write
\begin{equation*}
U_\e(t, \tau)z_\tau=U_{\e, 0}(t, \tau)z_\tau+ U_{\e, 1}(t, \tau)z_\tau,
\end{equation*}
where $U_{\e, 0}(t, \tau)z_\tau=(v_\e(t), v_{\e t}(t))$  solves
\begin{equation*}
{\left\{
  \begin{array}{ll}
    \rho_\e (t)v_{\e tt} +\alpha v_{\e t}+ A v_\e =0, \ \ t>\tau, \\
    U_{\e, 0}(\tau, \tau)z_\tau=z_\tau,
  \end{array}
\right.}
\end{equation*}
and $U_{\e, 1}(t, \tau)z_\tau=(w_\e(t), w_{\e t}(t))$ solves
\begin{equation*}
{\left\{
  \begin{array}{ll}
   \rho_\e (t) w_{\e tt} +\alpha w_{\e t}+ A w_\e +f(u_\e)=g,  \ \ t>\tau,  \\
    U_{\e, 1}(\tau, \tau)z_\tau=0.
  \end{array}
\right.}
\end{equation*}
 For any  $z_\tau\in \h_{\tau, 1/3}$ with  $\|z_\tau\|_{\h_{\tau, 1/3}}\leq R$, by Lemma \ref{lemma3} and estimate \eqref{a1}, we have
\begin{equation}\label{c10}
\sup_{\e\in [a, 1]}\|U_\e(t, \tau)z_\tau\|_{\h_{t, 1/3}}\leq \mathcal{Q}_a(R), \  \ t\geq \tau.
\end{equation}
Then repeating the same argument as Lemma 11.6 in \cite{pata}, we have

\begin{lemma}\label{lemma4}Let Assumption \ref{assumption51} be valid, and  $z_\tau\in \h_{\tau, 1/3}$ with $\|z_\tau\|_{\h_{\tau, 1/3}}\leq R$. Then
\begin{align}
&\sup_{\e\in [a, 1]} \|U_{\e, 0}(t, \tau)z_\tau\|_{\h_t}\leq \mathcal{Q}_a(R)e^{-\kappa (t-\tau)}, \label{c11}\\
& \sup_{\e\in [a, 1]}\|U_{\e, 1}(t, \tau)z_\tau\|_{\h_{t, 1}}\leq \mathcal{Q}_a(R),\ \ t\geq \tau.\label{c12}
\end{align}
\end{lemma}

Based on Lemma \ref{lemma4}, we  further give the desired  estimate.

\begin{lemma}\label{lemma5} Let Assumption \ref{assumption51} be valid, and  $z_\tau\in \h_{\tau,1}$ with $\|z_\tau\|_{\h_{\tau, 1/3}}\leq R$. Then
\begin{equation*}
\sup_{\e\in [a,1]}\|U_\e(t, \tau)z_\tau\|_{\h_{t, 1}}\leq \mathcal{Q}_a(\|z_\tau\|_{\h_{\tau, 1}})e^{-\kappa (t-\tau)}+\mathcal{Q}_a(R),\ \ t\geq \tau.
\end{equation*}
\end{lemma}
\begin{proof} For simplicity, we omit the subscript $\e$ and let $u=u_\e$. For any  $\e\in [a, 1]$ and $\delta\in (0, 1)$, we define the functional along the solution $(u(t), u_t(t))=U_\e(t, \tau)z_\tau$:
\begin{equation*}
\L_\e(t)=\|u\|^2_2+\rho_\e \|u_t\|_1^2-2\langle g, Au\rangle  +\delta\left[ 2\rho_\e\langle u_t, Au\rangle +\alpha \|u\|^2_1\right].
\end{equation*}
It follows from condition \eqref{5.4} that
\begin{align}
&2|\langle g, Au\rangle|\leq 2\|g\|\|u\|_2\leq \frac14\|u\|^2_2 + 4\|g\|^2, \nonumber\\
&2\rho_\e |\langle u_t, Au\rangle|\leq 2\rho_\e \|u_t\|_1 \|u\|_1\leq \alpha \|u\|^2_1+\frac{L}{\alpha} \rho_\e\|u_t\|^2_1. \label{c13}
\end{align}
Inserting \eqref{c13} into  $\L_\e(t)$ receives
\begin{equation}\label{c14}
\frac{a}{2}\|(u(t), u_t(t))\|^2_{\h_{t, 1}}- 4\|g\|^2\leq \L_\e(t)\leq 2\|(u(t), u_t(t))\|^2_{\h_{t, 1}}+ 4\|g\|^2
\end{equation}
for $\delta>0$ suitably small.
\medskip

Taking the multiplier $2A u_t +2\delta Au$ in Eq. \eqref{5.7} gives
\begin{equation}\label{c15}
\begin{split}
&\frac{\d}{\d t} \L_\e(t) +\delta \L_\e(t)+\left(2\alpha-\rho'_\e-3\delta \rho_\e\right)\|u_t\|^2_1+\delta\|u\|^2_2-\delta^2\alpha\|u\|^2_1\\
=\ &\left(2\delta^2 \rho_\e+ 2\delta \rho'_\e\right)\langle u_t, Au\rangle-\langle f(u), 2A u_t +2\delta Au\rangle,\ \ t>\tau.
\end{split}
\end{equation}
Exploiting the Sobolev embedding \eqref{c8}, condition \eqref{5.5} and estimate \eqref{c10}, we have
\begin{equation*}
\begin{split}
\|f(u)\|^2_1&=\|f'(u)A^{\frac12} u\|^2\\
&\leq C\int_\Omega \left(1+|u|^4\right) |A^{\frac12} u|^2\d x\\
&\leq C\left(1+\|u\|^4_{L^{18}(\Omega)}\right)\|A^{\frac12} u\|^2_{L^{\frac{18}{7}}(\Omega)}\\
&\leq C\left(1+\|u\|^4_{4/3} \right)\|u\|^2_{4/3}\leq \mathcal{Q}_a(R),
\end{split}
\end{equation*}
and hence,
\begin{equation*}
-\langle f(u), 2A u_t +2\delta Au\rangle\leq 2\|f(u)\|_1\left[\|u_t\|_1+ \|u\|_1\right]\leq \alpha\|u_t\|^2_1+\frac\delta 4 \|u\|^2_2 +\mathcal{Q}_a(R).
\end{equation*}

Due to
\begin{align*} &
\left(2\alpha-\rho'_\e-3\delta \rho_\e\right)\|u_t\|^2_1+\delta\|u\|^2_2-\delta^2\alpha\|u\|^2_1\geq \frac32 \alpha\|u_t\|^2_1+\frac\delta 2\|u\|^2_2,\\
& \left(2\delta^2 \rho_\e+ 2\delta \rho'_\e\right)\langle u_t, Au\rangle\leq 4\delta L\|u_t\|\|u\|_2\leq \frac\delta 4\|u\|^2_2+\mathcal{Q}_a(R)
\end{align*}
for $\delta>0$ suitably small, inserting above estimates into \eqref{c15}, we have
\begin{equation}\label{c0}
\frac{\d}{\d t}\L_\e(t)+\delta \L_\e(t)\leq \mathcal{Q}_a(R), \ \ t>\tau.
\end{equation}
 Applying the  Gronwall inequality to  \eqref{c0} and exploiting   estimate  \eqref{c14} turn out  the conclusion of Lemma \ref{lemma5}.
\end{proof}

\begin{lemma}\label{lemma6}  Let Assumption \ref{assumption51} be valid,  and  $z_{i\tau}\in \h_{\tau, 1}$ with $\|z_{i\tau}\|_{\h_{\tau, 1}}\leq R, i=1, 2$. Then, for all  $\e\in [a, 1]$, we have
\begin{equation*}
\|U_\e(t, \tau)z_{1\tau}- U_\e(t, \tau)z_{2,\tau}\|_{\h_t}\leq \mathcal{Q}_a(R)e^{-\kappa (t-\tau)}\|z_{1\tau}-z_{2\tau}\|_{\h_\tau}+\mathcal{Q}_a(R)\sup_{s\in [\tau, t]}\|\bar u_\e(s)\|,\ \ t\geq \tau,
\end{equation*}
 where
\begin{equation*}
(\bar u_\e(t), \bar u_{\e t}(t))=(u_{\e 1}(t), u_{\e 1t}(t))- (u_{\e 2}(t), u_{\e 2t}(t))=U_\e(t, \tau)z_{1\tau}- U_\e(t, \tau)z_{2\tau}.
\end{equation*}
\end{lemma}
\begin{proof}
It follows from Lemma \ref{lemma5} that
\begin{equation}\label{c19}
\sup_{\e\in[a, 1]}\|U_\e(t, \tau)z_{1\tau}\|_{\h_{t, 1}}+\|U_\e(t, \tau)z_{2\tau}\|_{\h_{t, 1}}\leq \mathcal{Q}_a(R), \ \ \forall t\geq \tau.
\end{equation}
For simplicity, we omit the subscript $\e$ and let $u_i=u_{i\e}, i=1,2$. Obviously, the difference $\bar u$ solves
\begin{equation}\label{c16}
\left\{
  \begin{array}{ll}
    \rho_\e (t)\bar u_{tt}+\alpha \bar u_t +A \bar u+ f(u_1)-f(u_2)=0, t> \tau,\\
    (\bar u(\tau), \bar u_t(\tau))= z_{1\tau}-z_{2\tau}.
  \end{array}
\right.
\end{equation}
Taking the  multiplier $2\bar u_t+ 2 \delta \bar u$ in Eq. \eqref{c16} gives
\begin{equation}\label{c17}
\begin{split}
& \frac{\d}{\d t} \Psi_\e(t) +2\delta \|\bar u\|^2_1+\left(2\alpha - \rho'_\e -2\delta \rho_\e\right) \|\bar u_t\|^2\\
=\ & 2\delta \rho'_\e \langle \bar u_t, \bar u\rangle-\langle f(u_1)-f(u_2), 2\bar u_t+ 2 \delta \bar u\rangle,\ \ t> \tau,
\end{split}
\end{equation}
where
\begin{align}\label{c18}&
\Psi_\e(t)=\rho_\e \|\bar u_t\|^2+ \|\bar u\|^2_1+\delta\left[2\rho_\e \langle \bar u_t, \bar u\rangle +\alpha \|\bar u\|^2\right]\sim \|(\bar u, \bar u_t)\|^2_{\h_t},\\
&\left(2\alpha - \rho'_\e -2\delta \rho_\e\right) \|\bar u_t\|^2\geq \alpha\|\bar u_t\|^2,\nonumber\\
&2\delta \rho'_\e \langle \bar u_t, \bar u\rangle\leq \frac\alpha4\|\bar u_t\|^2+\mathcal{Q}_a(R)\|\bar u\|^2\nonumber
\end{align}
for $\delta>0$ suitably small, and where we have used condition \eqref{5.4} and
formula \eqref{a1}.
By the Sobolev embedding $V_2 \hookrightarrow L^\infty(\Omega)$,  condition \eqref{5.5} and estimate \eqref{c19}, we have
\begin{equation*}
\begin{split}
& -\langle f(u_1)-f(u_2), 2\bar u_t+ 2 \delta \bar u\rangle\\
\leq \ & C\int_\Omega\left(1+ |u_1|^3+ |u_2|^3\right)|\bar u| \left[|\bar u_t|+|\bar u|\right]\d x\\
\leq \ & C\left( 1+ \|u_1\|^2_{L^\infty(\Omega)}+ \|u_2\|^2_{L^\infty(\Omega)}\right) \left[\|\bar u\|^2+ \|\bar u\|\|\bar u_t\|\right]\\
\leq \ & \mathcal{Q}_a(R)\|\bar u\|^2+\frac\alpha 4\|\bar u_t\|^2.
\end{split}
\end{equation*}
Inserting above estimates into \eqref{c17} and making use of \eqref{c18} yield
\begin{equation}\label{c20}
 \frac{\d}{\d t} \Psi_\e(t) +\delta \Psi_\e(t)\leq \mathcal{Q}_a(R)\|\bar u\|^2, \ \ t>\tau.
\end{equation}
Applying the Gronwall inequality to \eqref{c20} and making use of  formula \eqref{c18} receive
\begin{equation*}
\|(\bar u(t), \bar u_t(t))\|_{\h_t}\leq \mathcal{Q}_a(R)e^{-\kappa (t-\tau)}\|z_{1\tau}-z_{2\tau}\|_{\h_\tau}+\mathcal{Q}_a(R)\sup_{s\in [\tau, t]}\|\bar u(s)\|,\ \ t\geq\tau.
\end{equation*}
This completes the proof.
\end{proof}

\begin{lemma}\label{lemma7}  Let Assumption \ref{assumption51} be valid,  and $z_\tau\in \h_{\tau, 1}$ with $\|z_\tau\|_{\h_{\tau, 1}}\leq R$. Then for any  $\e_1, \e_2\in [a, 1]$,
\begin{equation}\label{c25}
\|U_{\e_1}(t, \tau)z_\tau- U_{\e_2}(t, \tau)z_\tau\|_{\h_t}\leq \mathcal{Q}_a(R)e^{\mathcal{Q}_a(R) (t-\tau)}|\e_1-\e_2|, \ \ \forall t\geq \tau.
\end{equation}
\end{lemma}
\begin{proof}  It follows from Lemma \ref{lemma5} that
\begin{equation}\label{c23}
\|U_{\e_1}(t, \tau)z_\tau\|_{\h_{t, 1}}+\|U_{\e_2}(t, \tau)z_\tau\|_{\h_{t, 1}}\leq \mathcal{Q}_a(R), \ \ \forall t\geq \tau.
\end{equation}
Let
\begin{equation*}
(u_i(t), u_{it}(t))= U_{\e_i}(t, \tau)z_\tau, \ \ t\geq \tau, \ i=1, 2.
\end{equation*}
Obviously,  the difference $\bar u= u_1-u_2$ solves
\begin{equation}\label{c21}
 \left\{
   \begin{array}{ll}
     \rho_{\e_2} \bar u_{tt}+ \left( \rho_{\e_1}- \rho_{\e_2}\right)u_{1tt}+\alpha \bar u_t + A\bar u+ f(u_1)-f(u_2)=0,\ \ t> \tau, \\
     \left(\bar u(\tau), \bar u_t(\tau)\right)=0.
   \end{array}
 \right.
\end{equation}
Taking the multiplier $2 \bar u_t$ in  Eq. \eqref{c21}   gives
\begin{equation}\label{c22}
\frac{\d}{\d t}\|(\bar u, \bar u_t)\|^2_{\h^{\e_2}_t}+\left(2\alpha- \rho'_{\e_2}\right)\|\bar u_t\|^2=- \left( \rho_{\e_1}- \rho_{\e_2}\right)\langle u_{1tt}, 2 \bar u_t\rangle
- \langle f(u_1)-f(u_2), 2 \bar u_t\rangle.
\end{equation}
Taking into account $\e_1, \e_2\in [a, 1]$,
by Eq. \eqref{5.7} and formula \eqref{c23}  we have
\begin{equation}\label{c24}
\rho(t)\|u_{1tt}\|\leq \frac{1}{a}\left[\alpha\|u_{1t}\|+\|Au_1\|+\|f(u_1)\|+\|g\|\right]\leq \mathcal{Q}_a(R), \ \ \forall t\geq \tau,
\end{equation}
where we have used the fact
$
\|f(u_1)\|\leq \mathcal{Q}_a(R)$ for $V_2\hookrightarrow L^\infty(\Omega)$.
It follows from estimates \eqref{c23} and \eqref{c24} that
\begin{align*}
& - \left( \rho_{\e_1}- \rho_{\e_2}\right)\langle u_{1tt}, 2 \bar u_t\rangle
\leq 2|\e_1-\e_2| \rho(t)\|u_{1tt}\|\|\bar u_t\|\leq \frac{\alpha}{2}\|\bar u_t\|^2+ \mathcal{Q}_a(R)|\e_1-\e_2|^2,\\
& - \langle f(u_1)-f(u_2), 2 \bar u_t\rangle
\leq \mathcal{Q}_a(R)
\|\bar u\|\|\bar u_t\|\leq \frac{\alpha}{2}\|\bar u_t\|^2+\mathcal{Q}_a(R)\|\bar u\|^2_1.
\end{align*}
Inserting above estimates into formula \eqref{c22} and making use of the fact $\rho'_\e<0$, we obtain
\begin{equation*}
\frac{\d}{\d t}\|(\bar u, \bar u_t)\|^2_{\h^{\e_2}_t}\leq \mathcal{Q}_a(R)\|(\bar u, \bar u_t)\|^2_{\h^{\e_2}_t}+\mathcal{Q}_a(R)|\e_1-\e_2|^2,\ \ t>\tau.
\end{equation*}
Applying the    Gronwall inequality  and estimate \eqref{a1} give   formula \eqref{c25}.
\end{proof}

\subsection{Proof of the main results}

Based on the technical preparation in  previous subsection, we prove Theorem \ref{52}  and Theorem \ref{54} by applying the abstract criteria obtained in Section 3 and Section 4, respectively. These arguments  are  challenging because of the hyperbolicity of model \eqref{5.7}.  The method developed here allows to  overcome this difficulty. We first  construct a desired family $\mathcal{B}$ belonging to  universe $\mathscr D$.

\begin{lemma}\label{lemma8} Let Assumption \ref{assumption51} be valid, and the interval $[a, 1]\subset (0, 1]$. Then   there exists a family $\mathcal{B}=\{B(t)\}_{t\in\mathbb R}\in \mathscr D$  possessing the following properties:
\begin{description}
  \item (i)  each section $B(t)$ is closed in $\h_t$ and
        \begin{equation}\label{d1}
        B(t)\subset \mathbb B_t(\mathcal R_0)\cap \mathbb B^1_t(\mathcal R), \ \ \forall t\in \mathbb{R}
         \end{equation}
       for some constants $\mathcal R=\mathcal R(a)>0$  and $\mathcal R_0>R_1$, where $R_1$ is as shown in Remark \ref{remark54}, $\mathbb B^1_t(\mathcal R)$ is the $\mathcal R$-ball in $\h_{t, 1}$  centered at  $0$;
  \item (ii)  there exist positive constants $\kappa_1$ and $\tau_1$ such that
   \begin{equation}\label{d2}
  \sup_{\e\in [a, 1]} \mathrm{dist}_{\h_t}\left(U_\e(t, \tau)\mathbb B_\tau(R_1), B(t)\right)\leq \mathcal Q_a(R_1)e^{-\kappa_1(t-\tau)}, \ \ \forall t\geq \tau+\tau_1,
   \end{equation}
   where the family  $\{\mathbb B_t(R_1)\}_{t\in\mathbb R}\in \mathscr D$ is as shown in Remark \ref{remark54};
   \item (iii) there exists a positive constant $T_1$ such that
    \begin{equation}\label{a5.3}
      \bigcup_{\e\in [a, 1]}U_\e(t, \tau)B(\tau)\subset B(t), \ \ \forall t\geq \tau+T_1.
      \end{equation}
\end{description}
\end{lemma}
\begin{proof} For any  $z_\tau\in \mathbb B_\tau(R_1)$,  it follows from Lemma \ref{lemma2} that
\begin{align*}
& \sup_{\e\in [a, 1]}\|U_{\e,0}(t, \tau)z_\tau\|^2_{\h_t}\leq \mathcal Q_a(R_1)e^{-\kappa (t-\tau)}\ \ \hbox{and}\ \   \sup_{\e\in [a, 1]}\|U_{\e, 1}(t, \tau)z_\tau\|^2_{\h_{t,1/3}}\leq \mathcal Q_a(R_1),\ \ \forall t\geq \tau,
\end{align*}
  which imply that there exists a positive constant $\mathcal R_1=\mathcal R_1(R_1)$  such that
 \begin{equation}\label{a5.4}
  \sup_{\e\in [a, 1]} \mathrm{dist}_{\h_t}\left(U_\e(t, \tau)\mathbb B_\tau(R_1),\  \mathbb B^{1/3}_t(\mathcal R_1)\right)\leq \mathcal Q_a(R_1)e^{-\kappa(t-\tau)}, \ \ \forall t\geq \tau.
   \end{equation}
Similarly, for any $z_\tau\in \mathbb B^{1/3}_\tau(\mathcal R_1)$, we infer from Lemma \ref{lemma4} that
\begin{align*}
&\sup_{\e\in [a, 1]}\|U_{\e,0}(t, \tau)z_\tau\|^2_{\h_t}\leq \mathcal Q_a(\mathcal R_1) e^{-\kappa(t-\tau)}\ \ \hbox{and}\ \  \sup_{\e\in [a, 1]}\|U_{\e,1}(t, \tau)z_\tau\|^2_{\h_{t,1}}\leq \mathcal Q_a(\mathcal R_1),\ \ t\geq \tau,
\end{align*}
  which imply that there exists a positive constant $\mathcal R_2=\mathcal R_2(R_1)$  such that
\begin{equation}\label{a5.5}
  \sup_{\e\in [a, 1]} \mathrm{dist}_{\h_t}\left(U_\e(t, \tau)\mathbb B^{1/3}_\tau(\mathcal R_1),\  \mathbb B^1_t(\mathcal R_2)\right)\leq \mathcal Q_a(R_1)e^{-\kappa(t-\tau)}, \ \ \forall t\geq \tau.
   \end{equation}
Remark \ref{remark54} shows that the family  $\{\mathbb B_t(R_1)\}_{t\in\mathbb R}\in \mathscr D$  is a uniformly (w.r.t. $\e\in [a, 1]$) pullback $\mathscr D$-absorbing family of the processes $U_\e(t, \tau), \e\in [a,1]$, that is,
 there exists a positive constant $e(R_1)$ such that
\begin{equation}\label{a5.6}
\bigcup_{\e\in [a, 1]}U_\e(t, \tau)\mathbb B_\tau (R_1)\subset \mathbb B_t(R_1), \ \ \forall t\geq \tau+e(R_1).
\end{equation}

 Let $\theta=\frac{\kappa}{\mathcal Q_a(R_1)+2\kappa}$. Obviously,
\begin{equation*}
\theta\in (0, 1)\ \ \hbox{and}\ \ -\kappa \theta=-\kappa+ \left(\mathcal Q_a(R_1)+\kappa\right)\theta.
\end{equation*}
We infer from formula \eqref{a5.6} that there exists a constant $e_1=\frac{e(R_1)}{1-\theta}>0$ such that
\begin{equation}\label{a5.7}
\bigcup_{\e\in [a, 1]}U_\e\left((1-\theta)t+\theta\tau, \tau\right)\mathbb B_\tau (R_1)\subset \mathbb B_{(1-\theta)t+\theta\tau}(R_1), \ \ \forall t\geq \tau+e_1.
\end{equation}
For any $t\geq \tau$, let $t_1=(1-\theta)t+\theta\tau$.  It follows from Lemma \ref{lemma6} and formulas \eqref{a5.4}-\eqref{a5.7} that
\begin{align}
& \sup_{\e\in [a, 1]}\mathrm{dist}_{\h_t}\left(U_\e(t, \tau)\mathbb B_\tau (R_1), \mathbb B_t^1(\mathcal R_2)\right)\nonumber\\
\leq \ & \sup_{\e\in [a, 1]}\mathrm{dist}_{\h_t}\left(U_\e\left(t, t_1\right)U_\e\left(t_1, \tau\right)\mathbb B_\tau (R_1),  U_\e\left(t, t_1\right)\mathbb B_{t_1}^{1/3}(\mathcal R_1)\right)\nonumber\\
&+  \sup_{\e\in [a, 1]}\mathrm{dist}_{\h_t}\left(U_\e\left(t, t_1\right)\mathbb B_{t_1}^{1/3}(\mathcal R_1),  \mathbb B_t^1(\mathcal R_2) \right)\label{a5.8}\\
\leq\ &\mathcal Q_a(R_1) \exp\{\mathcal Q_a(R_1)(t-t_1)\} \sup_{\e\in [a, 1]}\mathrm{dist}_{\h_{t_1}}\left(U_\e\left(t_1, \tau\right)\mathbb B_\tau (R_1),  \mathbb B_{t_1}^{1/3}(\mathcal R_1)\right)\nonumber\\
&+ \mathcal Q_a(R_1)e^{-\kappa  (t-t_1)}\nonumber\\
\leq \ & \mathcal Q_a(R_1)\exp\{\left[-\kappa+ \left(\mathcal Q_a(R_1)+ \kappa\right)\theta \right](t-\tau)\}+ \mathcal Q_a(R_1)e^{-\kappa \theta (t-\tau)}\nonumber\\
\leq \ & \mathcal Q_a(R_1)e^{-\kappa \theta (t-\tau)},\ \ \forall  t\geq \tau+e_1.\nonumber
\end{align}

For every $\xi\in \mathbb B_t^1(\mathcal R_2)$,
\begin{equation*}
\|\xi\|_{\h_t}\leq \lambda_1^{-1/2} \|\xi\|_{\h_{t,1}}\leq \lambda_1^{-1/2} \mathcal R_2,\ \ \forall t\in\mathbb R,
\end{equation*}
which implies
\begin{equation}\label{a5.9}
\mathbb B_t^1\left(\mathcal R_2\right)\subset \mathbb B_t\left(\lambda_1^{-1/2}\mathcal R_2\right) \subset \mathbb B_t(\mathcal R_3) \ \ \hbox{and}\ \ \mathbb B_t(R_1) \subset \mathbb B_t(\mathcal R_3),\ \ \forall  t\in\mathbb R,
\end{equation}
 with $\mathcal R_3=R_1+ \lambda_1^{-1/2}\mathcal R_2(R_1)=\mathcal R_3(R_1)$. By Remark  \ref{remark54} and formula \eqref{a5.9}, there exists a constant $e_2=e_2(R_1)>0$ such that
\begin{equation}\label{a5.10}
 \bigcup_{\e\in [a, 1]} U_\e(t, \tau)\mathbb B_\tau(\mathcal R_3)\subset\mathbb B_t(R_1) \subset \mathbb B_t(\mathcal R_3), \ \ \forall  t\geq \tau+e_2.
\end{equation}
It follows from Lemma \ref{lemma3} that for any  $z_\tau\in \mathbb B_\tau(\mathcal R_3)\cap \h_{\tau, 1/3}$,
\begin{equation}\label{a5.11}
\sup_{\e\in [a, 1]}\|U_\e(t, \tau)z_\tau\|^2_{\h_{t, 1/3}}\leq \mathcal Q_a\left(\mathcal R_3+\|z_\tau\|_{\h_{\tau, 1/3}}\right)e^{-\kappa(t-\tau)}+\mathcal R_4, \ \ \forall t\geq \tau,
\end{equation}
with the positive constant $\mathcal R_4=Q_a(\mathcal R_3(R_1))=\mathcal R_4(R_1)$.

Similarly, for every  $\xi\in \mathbb B_t^1(\mathcal R_2)$, we have
\begin{equation*}
\|\xi\|_{\h_{t, 1/3}}\leq \lambda_1^{-1/3} \|\xi\|_{\h_{t,1}}\leq \lambda_1^{-1/3}\mathcal R_2, \ \ \forall t\in\mathbb R,
\end{equation*}
which means
\begin{equation}\label{a5.12}
\mathbb B_t^1\left(\mathcal R_2\right)\subset \mathbb B^{1/3}_t\left(\lambda_1^{-1/3} \mathcal R_2\right) \subset \mathbb B^{1/3}_t(\mathcal R_5),\ \ \forall t\in\mathbb R,
\end{equation}
where $\mathcal R_5=\mathcal R_4+ \lambda_1^{-1/3} \mathcal R_2=\mathcal R_5(R_1)$. It follows from  formula \eqref{a5.11} that there exists a positive constant $e_3=e_3(R_1)$ such that
\begin{equation}\label{a5.13}
 \bigcup_{\e\in [a, 1]}U_\e(t, \tau)\left[\mathbb B_\tau\left(\mathcal R_3\right) \cap \mathbb B^{1/3}_\tau(\mathcal R_5) \right]\subset \mathbb B^{1/3}_t(\mathcal R_5),  \ \ \forall  t\geq \tau+e_3.
\end{equation}

Lemma  \ref{lemma5} shows that for any $z_\tau\in
\mathbb B^{1/3}_\tau(\mathcal R_5)\cap \h_{\tau,1}$,
\begin{equation}\label{a5.14}
\sup_{\e\in [a, 1]}\|U_\e(t, \tau)z_\tau\|^2_{\h_{t,1}}\leq  \mathcal Q_a\left(\mathcal R_5+\|z_\tau\|_{\h_{\tau,1}}\right)e^{-\kappa(t-\tau)}+\mathcal R_6, \ \ \forall t\geq \tau,
\end{equation}
where the  positive constant $\mathcal R_6=\mathcal Q_a(\mathcal R_5)=\mathcal R_6(R_1)$. Obviously,
\begin{equation}\label{a5.15}
\mathbb B^1_t(\mathcal R_2)\subset \mathbb B^1_t(\mathcal R_7)\ \ \hbox{with}\ \ \mathcal R_7=\mathcal R_2+ \mathcal R_6=\mathcal R_7(R_1), \ \ \forall t\in\mathbb R.
\end{equation}
Formula \eqref{a5.14} implies that there exists a positive constant $e_4=e_4(R_1)$ such that
\begin{equation}\label{a5.16}
\bigcup_{\e\in [a, 1]}U_\e(t, \tau)\left[\mathbb B^{1/3}_\tau\left(\mathcal R_5\right) \cap \mathbb B^1_\tau(\mathcal R_7) \right]\subset \mathbb B^1_t(\mathcal R_7),\ \ \forall  t\geq \tau+e_4.
\end{equation}

Let
\begin{equation*}
B(t)=\mathbb B_t\left(\mathcal R_3\right) \cap \mathbb B^{1/3}_t(\mathcal R_5) \cap \mathbb B^1_t(\mathcal R_7), \ \ \forall t\in\mathbb R.
\end{equation*}
We show that $\{B(t)\}_{t\in\mathbb R}\in \mathscr D$ is the desired family.\medskip

(i)\ Obviously, for every $t\in\mathbb R$, $B(t)$ is closed in $\h_t$ and
  \begin{equation*}
B(t)\subset \mathbb B_t\left(\mathcal R_3\right)  \cap \mathbb B^1_t(\mathcal R_7),\ \ \forall t\in \mathbb{R},
\end{equation*}
that is, formula \eqref{d1} holds, with $\mathcal R_0= \mathcal R_3> R_1$ and $\mathcal R=\mathcal R_7=\mathcal R(a)$.

(ii)\ It follows from formulas \eqref{a5.9}, \eqref{a5.12} and \eqref{a5.15} that
$\mathbb B^1_t(\mathcal R_2)\subset B(t)$  for all $t\in\mathbb R$. Then  we infer from estimate \eqref{a5.8} that
\begin{equation*}
\begin{split}
\sup_{\e\in [a, 1]}\mathrm{dist}_{\h_t}\left(U_\e(t, \tau)\mathbb B_\tau (R_1), B(t)\right)
&\leq \sup_{\e\in [a, 1]}\mathrm{dist}_{\h_t}\left(U_\e(t, \tau)\mathbb B_\tau (R_1), \mathbb B^1_t(\mathcal R_2)\right)\\
&\leq \mathcal Q_a(R_1) e^{-\kappa\theta (t-\tau)}, \ \ \forall t\geq \tau +e_1,
\end{split}
\end{equation*}
that is, formula \eqref{d2} holds, with   $\kappa_1=\kappa \theta$ and $\tau_1=e_1$.

(iii)\ Taking $T_1=\max\{e_2, e_3, e_4\}$ and making use of formulas \eqref{a5.10}, \eqref{a5.13} and \eqref{a5.16} yield
 \begin{align*}
\bigcup_{\e\in [a, 1]}U_\e(t, \tau)B(\tau)&\subset \begin{cases} &\bigcup_{\e\in [a, 1]}U_\e(t, \tau)\mathbb B_\tau (\mathcal R_3)\subset \mathbb B_t (\mathcal R_3),\\
&  \bigcup_{\e\in [a, 1]}U_\e(t, \tau)\left[\mathbb B_\tau (\mathcal R_3)\cap
\mathbb B^{1/3}_\tau \left(\mathcal R_5\right)\right]\subset \mathbb B^{1/3}_t\left(\mathcal R_5\right),\\
& \bigcup_{\e\in [a, 1]}U_\e(t, \tau)\left[\mathbb B^{1/3}_\tau (\mathcal R_5)\cap \mathbb B^1_\tau \left(\mathcal R_7\right)\right]\subset \mathbb B^1_t\left(\mathcal R_7\right),
\end{cases}
\end{align*}
for all $t\geq \tau+T_1$.  Therefore,
\begin{equation*}
\bigcup_{\e\in [a, 1]}U_\e(t, \tau)B(\tau)\subset \mathbb B_t\left(\mathcal R_3\right) \cap \mathbb B^{1/3}_t(\mathcal R_5) \cap \mathbb B^1_t(\mathcal R_7)=B(t), \ \ \forall t\geq \tau+T_1.
\end{equation*}
This completes the proof.
\end{proof}

\begin{proof}[\textbf{Proof of Theorem \ref{54}}] For any $\e_0\in (0, 1]$, there must be an interval $[a, 1]\subset (0, 1]$ such that $\e_0\in [a, 1]$.   Lemma \ref{lemma8} shows that there exists a family $\mathcal B=\{B(t)\}_{t\in \mathbb R}\in \mathscr D$, with the properties (i)-(iii) there.

 Take $T> T_1$ satisfying $\eta:=\q_a(\mathcal R)e^{-\kappa T}<1/4$,  formulas \eqref{d1} and \eqref{a5.3} mean that condition $(H_1)$ of Assumption \ref{assumptionea} holds, with $X_t=\h_t$ and $\Lambda=[a,1]$ there.

Let the  space
\begin{equation*}
Z=\{\phi\in C([0, T]; V_1)\ | \ \phi_t\in C([0, T]; L^2)\}
\end{equation*}
be equipped with the  norm
\begin{equation*}
  \|\phi\|_Z=\sup_{s\in[0, T]}\|(\phi(s),\phi_t(s))\|_{V_1\times L^2}.
\end{equation*}
Obviously, $Z$ is a Banach space, and
\[ n_Z(\phi)=\q_a(\mathcal R)\sup_{s\in[0, T]}\|\phi(s)\|, \ \ \forall \phi\in Z\]
 is a compact seminorm on $Z$ (cf. \cite{Simon}).  Taking $t_0=0$, we define the mapping
\begin{equation*}
 K^\e_t: B(t-T)\subset \h_{t-T}\rightarrow Z, \ \ K^\e_t\xi=u_\e(\cdot+t-T),\ \ \forall \xi\in B(t-T),\ \ \e\in [a, 1],\ \ t\leq t_0,
\end{equation*}
where $u_\e(\cdot+t-T) $ means $ u_\e(s+t-T), s\in[0,T]$ and
\begin{equation*}
(u_\e(s+t-T),u_{\e t}(s+t-T))=U_\e (s+t-T, t-T)\xi.
\end{equation*}
Formula \eqref{4.10} shows that for all $\xi_1, \xi_2\in B(t-\tau)(\subset \mathbb B_{t-\tau}(\mathcal R_0)\cap \mathbb B^1_{t-\tau}(\mathcal R))$, $\tau\in [0, T]$,
\begin{equation*}
\sup_{\e\in [a,1]}\|U_\e(t, t-\tau)\xi_1- U_\e(t, t-\tau)\xi_2\|_{\h_t}\leq L_1\|\xi_1-\xi_2\|_{\h_{t-\tau}},\ \ \forall t\in \mathbb{R},
\end{equation*}
with $L_1=\q_a(\mathcal R_0)e^{\q_a(\mathcal R_0) T}$, and Lemma \ref{lemma6} shows that when
  $t\leq t_0=0$,
\begin{equation*}
\|U_\e(t, t-T)\xi_1- U_\e(t, t-T)\xi_2\|_{\h_t}\leq \eta\|\xi_1-\xi_2\|_{\h_{t-T}}+n_Z\left( K^\e_t\xi_1- K^\e_t\xi_2 \right),
\end{equation*}
where
\begin{equation*}
\begin{split}
    & \sup_{\e\in [a,1]}\|K^\e_t \xi_1- K^\e_t \xi_2\|_Z\\
=\  & \sup_{\e\in [a,1]}\sup_{s\in [0,T]} \|U_\e (s+t-T, t-T)\xi_1- U_\e (s+t-T, t-T)\xi_2\|_{V_1\times L^2}\\
\leq\ &  \sup_{s\in [0,T]}\left[ 1+\left(\rho(s+t-T)\right)^{-\frac12}\right]\sup_{\e\in [a,1]} \|U_\e (s+t-T, t-T)\xi_1- U_\e (s+t-T, t-T)\xi_2\|_{\h_{s+t-T}}\\
\leq\ & \left[ 1+\left(\rho(0)\right)^{-\frac12}\right]\q_a(\mathcal R)e^{\q_a(\mathcal R) T}\|\xi_1-\xi_2\|_{\h_{t-T}},
\end{split}
\end{equation*}
where we have used the fact that $\rho$ is a decreasing function. That is, the conditions $(H_2)$-$(H_3)$ of Assumption \ref{assumptionea} hold.

Moreover,  Remark \ref{remark54} shows that  the process $U_\lambda(t, \tau)$  has  a uniformly pullback $\mathscr D$-absorbing family $\mathcal D_0=\{\mathbb{B}_t(R_1)\}_{t\in\mathbb{R}}\in \mathscr D$  possessing the properties: (see \eqref{d2}, \eqref{d1} and \eqref{4.10})
\begin{enumerate}[(i)]
  \item
   \begin{equation}\label{14.8}
   \sup_{\e\in [a,1]} \mathrm{dist}_{X_t}\left(U_\e(t, t-\tau)\mathbb{B}_{t-\tau}(R_1), B(t)\right)\leq  \mathcal Q_a(R_1)e^{-\kappa_1\tau}, \ \ \forall t\in \mathbb R, \ \tau\geq \tau_1;
   \end{equation}
  \item  $\mathbb{B}_t(R_1)\subset \mathbb B_t(\mathcal R_0), \forall t\in \mathbb R$ for $\mathcal R_0 >R_1$, and
      \begin{equation}\label{14.9}
      \sup_{\e\in [a, 1]}\|U_\e(t, t-\tau)\xi_1- U_\e(t, t-\tau)\xi_2\|_{\h_t}\leq \mathcal{Q}_a(\mathcal R_0)e^{\mathcal{Q}_a(\mathcal R_0)\tau}\|\xi_1-\xi_2\|_{\h_{t-\tau}},
      \end{equation}
       for all $x, y\in \mathbb B_{t-\tau}(\mathcal R_0)$, $\tau\geq 0$ and $t\in \mathbb R$.
     \end{enumerate}
 That is, the conditions of Corollary \ref{ea2} hold. Therefore,     by  Corollary \ref{ea2}, the process $U_\e(t, \tau)$ has a pullback $\mathscr D$-exponential attractor $\mathcal E_\e=\{E_\e(t)\}_{t\in\mathbb R}$  for each $\e\in [a, 1]$, and   $E_\e(t)\subset B(t)\subset \mathbb B^1_t(\mathcal R)$ for all $t\in \mathbb R$.  By the arbitrariness $a\in (0,1)$ we have
\[\sup_{t\in \mathbb R}\|E_\e(t)\|_{\mathcal {H}_{t, 1}}< +\infty, \ \ \forall \e\in(0,1].\]

  Lemma \ref{lemma7} shows that
\begin{equation*}
\Gamma(\e, \e_0):= \sup_{t\leq t_0} \sup_{s\in [0, T]} \sup_{\xi\in B(t-T)}\|U_\e (t, t-s)\xi- U_{\e_0} (t, t-s)\xi\|_{\h_t}\leq \q_a(\mathcal R)e^{\q_a(\mathcal R) T} |\e-\e_0|.
\end{equation*}
Taking $\delta=\delta(\e_0)=\left(\q_a(\mathcal R)e^{\q_a(\mathcal R) T}\right)^{-1}$, we obtain that  $\Gamma(\e, \e_0)<1$ whenever  $|\e-\e_0|<\delta$. Then
  by  Theorem \ref{42},
\begin{equation*}
\mathrm{dist}_{\h_t}^{symm} \left(E_\e(t), E_{\e_0}(t)\right)\leq C(t)|\e-\e_0|^\gamma\ \ \hbox{as}\ \ |\e-\e_0|<\delta, \ \forall t\in \mathbb R.
\end{equation*}
where $\gamma\in (0, 1)$ is a positive constant. By the arbitrariness of $\e_0\in (0, 1]$, we complete  the proof.
\end{proof}

In order to prove Theorem \ref{52}, we first quote  a few notations   and lemmas on the    pullback $\mathscr D$-attractor.

\begin{definition}\label{asycompact} (Pullback $\mathscr D$-asymptotically compact)  A process $U(t, \tau): X_\tau\rightarrow X_t$ is said to be pullback $\mathscr D$-asymptotically compact if for any  $\{\tau_n\}\subset (-\infty, t]$ with $\tau_n\rightarrow -\infty$, and $y_n\in D(\tau_n)\subset \mathcal D  \in \mathscr D$, the sequence $\{U(t, \tau_n)y_n\}$ is relatively compact in $X_t$.
\end{definition}

\begin{lemma}\label{uniformlycompact}\cite{Kloeden2008} If the process $U(t, \tau): X_\tau\rightarrow X_t$ has a compact pullback $\mathscr D$-attracting family, then it is pullback $\mathscr D$-asymptotically compact.
\end{lemma}

\begin{lemma}\label{existTh}\cite{Kloeden2008} Assume that the  process $U(t, \tau): X_\tau\rightarrow X_t$ is continuous, and
 (i) it has  a pullback $\mathscr D$-absorbing family  $\mathcal B_0=\{B_0(t)\}_{t\in\mathbb R}\in \mathscr D$;
 (ii) it is pullback $\mathscr D$-asymptotically compact.
 Then the family $\mathcal A=\{A(t)\}_{t\in\mathbb R}$, with
\begin{equation*}
A(t):=\bigcap_{s\leq t}\Big[\bigcup_{\tau\leq s} U(t, \tau) B_0(\tau)\Big]_{X_t}, \ \ t\in \mathbb R,
\end{equation*}
is the minimal pullback $\mathscr D$-attractor of $U(t, \tau)$.
\end{lemma}

\begin{proof}[\textbf{Proof of Theorem \ref{52}}]For any $n\in\mathbb{N}^+=\{1,2,\cdots\}$, let
\begin{equation}\label{e1}
[a_n, 1]=[ 1/(n+1), 1].
\end{equation}
Obviously, $\cup_{n\in\mathbb{N}^+}[a_n, 1]=(0, 1]$.
\medskip

(i) Remark \ref{remark54} shows that the process  $U_\e(t, \tau)$ has a uniformly pullback $\mathscr D$-absorbing family $\{\mathbb B_t(R_1)\}_{t\in \mathbb R}$, that is,
 for any  $\mathcal D\in \mathscr D$, there exists a constant  $e(\mathcal D)>0$ such that
\begin{equation}\label{e2}
 \bigcup_{\e\in [a_n,1]} U_\e(t, t-\tau)D(t-\tau)\subset  \mathbb B_t(R_1), \ \ \forall \tau\geq e(\mathcal D), \  t\in \mathbb R.
\end{equation}
   Lemma \ref{lemma8} shows that there exists a   family $\mathcal B=\{B(t)\}_{t\in \mathbb R}\in \mathscr D$ possessing the properties (i)-(iii) there.
  Formula \eqref{d1} implies that the section $B(t)$ is compact in $\h_t$ for each $t\in \mathbb R$ because of  $\h_{t, 1}\hookrightarrow\hookrightarrow \h_t$.
The combination of \eqref{d2} and \eqref{e2} receives
\begin{equation*}
  \begin{split}
    &\sup_{\e\in [a_n,1]}\mathrm{dist}_{\h_t}\left(U_\e(t,t-\tau)D(t-\tau), B(t)\right)\\
  \leq\ & \sup_{\e\in [a_n,1]}\mathrm{dist}_{\h_t}\left(U_\e(t, t-\tau+e(\mathcal D))U_\e(t-\tau+e(\mathcal D), t-\tau)D(t-\tau), B(t)\right) \\
       \leq\ & \sup_{\e\in [a_n,1]}\mathrm{dist}_{\h_t}\left(U_\e(t, t-\tau+e(\mathcal D))\mathbb B_{t-\tau+e(\mathcal D)}(R_1), B(t)\right)\\
     \leq\ & \q_{a_n}(R_1)e^{\kappa_1 e(\mathcal D)}e^{-\kappa_1\tau}, \ \ \forall \tau\geq e(\mathcal D)+\tau_1, \ t\in\mathbb{R},
  \end{split}
\end{equation*}
that is, $\mathcal B=\{B(t)\}_{t\in \mathbb R}$ is a compact pullback $\mathscr D$-attracting family of the process $U_\e(t, \tau)$.   Therefore, by   Lemma \ref{uniformlycompact} and Lemma \ref{existTh},    the family $\mathcal A_\e=\{A_\e(t)\}_{t\in \mathbb R}$, with
\begin{equation*}
A_\e(t):=\bigcap_{s\leq t}\Big[\bigcup_{\tau\leq s} U_\e(t, \tau)\mathbb B_\tau(R_1)\Big]_{\h_t}, \ \ t\in \mathbb R
\end{equation*}
is the minimal pullback $\mathscr D$-attractor of the process $U_\e(t, \tau)$, and (see  \eqref{d1})
\begin{equation}\label{e3}
\bigcup_{\e\in [a_n,1]}A_\e(t)\subset B(t)\subset\mathbb B^1_t(\mathcal R), \ \ \forall t\in \mathbb R,
\end{equation}
where the constant $\mathcal R=\mathcal R(a_n)$. Therefore,  by the arbitrariness of $n\in \mathbb{N}^+$,
\begin{equation*}
    \sup_{t\in \mathbb R}\|A_\e(t)\|_{\mathcal {H}_{t, 1}}<\infty, \ \ \forall \e\in(0,1].
    \end{equation*}
   We know from  the minimality  of $\mathcal A_\e$ that $A_\e(t)\subset E_\e(t)$ for all $t\in \mathbb R$, where $\mathcal E_\e=\{E_\e(t)\}_{t\in \mathbb R}$  is a pullback $\mathscr D$-exponential attractor of the process $U_\e(t, \tau)$ as shown in  Theorem \ref{54}.  Therefore,
\begin{equation*}
\sup_{t\in \mathbb R} \mathrm{dim}_f\left(A_\e(t), \h_t\right)\leq \sup_{t\in \mathbb R} \mathrm{dim}_f\left(E_\e(t), \h_t\right)<+\infty, \ \ \forall \e\in(0,1].
\end{equation*}

(ii) (Upper semicontinuity) For any $\e_0\in (0,1]$, there must be  $\e_0\in \Lambda_n:=[a_n, 1]$ for some $n\in \mathbb{N}^+$.   For any $\xi\in B(t) (\subset \mathbb{B}^1_t(\mathcal{R}))$, by  Lemma \ref{lemma7}  we have
\begin{equation*}
 \sup_{\xi\in B(\tau)}\|U_\e(t, \tau)\xi-U_{\e_0}(t, \tau)\xi\|_{\h_t}\leq \q_{a_n}(\mathcal R)e^{\q_{a_n}(\mathcal R)(t-\tau)}|\e-\e_0|\rightarrow0\ \  \hbox{as }\ \ \e\rightarrow\e_0.
\end{equation*}
 And  $\mathcal A_\e$ pullback attracts $\mathcal{B}=\{B(t)\}_{t\in \mathbb R}$ for all $\e\in \Lambda_n$ for $\mathcal{B}\in \mathscr D$, that is,
 \begin{equation} \label{4.59}
 \lim_{\tau\rightarrow -\infty}\mathrm{dist}_{\h_t}\left(U_\e(t,\tau)B(\tau), A_\e(t)\right)=0, \  \ \forall t\in \mathbb{R}, \ \e\in \Lambda_n.
 \end{equation}
 Formulas \eqref{e3}-\eqref{4.59} mean  that the conditions $(L_1)$-$(L_3)$ of Assumption \ref{3.1'} hold for $\e\in \Lambda_n$. By Theorem \ref{31}, the pullback $\mathscr D$-attractor $\mathcal A_\e$ is upper semicontinuous at  $\e_0$. By the arbitrariness of $\e_0\in (0,1]$, $\mathcal A_\e$ is upper semicontinuous on  $(0,1]$.
\medskip

(iii) (Residual continuity) Since   $\Lambda_n=[a_n, 1]$ is a compact metric space for each $n\in \mathbb{N}^+$, we  infer from Theorem \ref{32} that there exists a residual subset $\Lambda^*_n\subset \Lambda_n$ such that $\mathcal A_\e$ is continuous on $\Lambda^*_n$. Let
\begin{equation*}
\Lambda^*=\bigcup_{n\in\mathbb{N}^+}\Lambda^*_n.
\end{equation*}
Then $\Lambda^*$ is still a residual subset of $(0,1]$ and $\mathcal A_\e$ is continuous on $\Lambda^*$ for the arbitrariness of $n\in\mathbb{N}^+$.
\end{proof}

\end{document}